\documentclass[12pt,reqno]{amsart}
\usepackage{amsfonts, amssymb, amsmath}
\usepackage{amsthm}                               % AMS symbols
\usepackage{setspace}                                                         % Spacing
\usepackage{geometry}                                                         % Spacing
\usepackage{color}                                                            % Text coloring
\usepackage{graphicx}                                                         % Graphics
\usepackage{slashed}                                                          % Slashed notations
\usepackage{hyperref}
\usepackage{float}
\usepackage{verbatim}
\usepackage{upgreek}
\usepackage{mathtools}
\usepackage[capitalise]{cleveref}
\usepackage{mathrsfs}
\newtheorem{theorem}{Theorem} 	      	      	                              % Theorem environment
\newtheorem{lemma}[theorem]{Lemma}     	       	      	      	      	      % Lemma environment
\newtheorem{proposition}[theorem]{Proposition} 	      	      	      	      % Proposition environment
\newtheorem{definition}[theorem]{Definition} 	      	      	                % Definition environment
\newtheorem{remark}[theorem]{Remark}                                          % Remark environment
\newtheorem{assumption}[theorem]{Assumption}    	      	                    % Assumption environment
\newcommand{\thistheoremname}{}
\newtheorem*{theorem*}{Theorem}
\newtheorem*{genericthm*}{\thistheoremname}
\newenvironment{namedthm*}[1]
  {\renewcommand{\thistheoremname}{#1}%
   \begin{genericthm*}}
  {\end{genericthm*}}

\numberwithin{equation}{section}                                              % Equation numbering
\numberwithin{theorem}{section}                                               % Theorem numbering
\numberwithin{figure}{section}                                                % Figure numbering

\newcommand{\mc}[1]{\mathcal{#1}}                                             % Mathcal
\newcommand\numberthis{\addtocounter{equation}{1}\tag{\theequation}}

\usepackage{marginnote}

\newcommand{\R}{\mathbb{R}}                                                   % Real numbers
\newcommand{\rd}{{\rm d}}

\newcommand{\namla}{\bar{\nabla}}                                             % Minkowski differential
\newcommand{\mdot}{\star}                                                  % Minkowski dot product
\newcommand{\nasla}{\slashed{\nabla}}                                         % Slashed nabla

\newcommand{\nt}{\numberthis}

\newcommand{\p}{\partial}
\newcommand{\nb}{\nabla}

\allowdisplaybreaks
\title[Observability for critically singular potentials]{Observability for wave equations with critically singular potentials}

\author{Vaibhav Kumar Jena}
\address{Department of Mathematics\\ 
Indian Institute of Science \\ Bengaluru 560012, Karnataka, India.}
\email{vaibhavkuma1@iisc.ac.in}

\author{Arick Shao}
\address{School of Mathematical Sciences\\
Queen Mary University of London\\
London E1 4NS\\ United Kingdom.}
\email{a.shao@qmul.ac.uk}

\begin{document}

\begin{abstract}
In this article, we prove boundary observability estimates for wave equations on bounded domains $\Omega \subseteq \mathbb{R}^n$, with a critically singular potential that diverges as the inverse square distance to $\partial \Omega$.
The result is applicable in all dimensions and for wave equations with general time-dependent lower order terms. The key geometric assumption is a convexity condition on $\Omega$.
The main tool for our result is a global Carleman estimate that is carefully adapted to the critical singular nature of the potential.
\end{abstract}

\maketitle

\section{Introduction}

In this paper, we consider the following setting:

\begin{assumption}[General setting] \label{ass.setting}
Let $n \geqslant 1$, and let $\Omega \subseteq \R^n$ be a bounded open subset with $C^4$-boundary $\Gamma := \partial \Omega$.
In addition, we let $d_\Gamma: \Omega \rightarrow \mathbb{R}$ denote the distance to $\Gamma$.

We also fix $T > 0$, and we consider the following wave equation on $( -T, T ) \times \Omega$,
\begin{equation} \label{eq_intro_wave}
- \p_t^2 u + \Delta_x u + \frac{\sigma}{d_\Gamma^2} u + X_t \partial_t u + X_x \cdot \nabla_x u + V u = 0 \text{,}
\end{equation}
where $\sigma \in ( -\frac{3}{4}, 0 )$, $V \in d_\Gamma^{-1} L^\infty (( -T, T ) \times \Omega; \R )$, and $X \in L^\infty ( ( -T, T ) \times \Omega; \R_t \times \R_x^n )$.
\end{assumption}

The key feature of our equation \eqref{eq_intro_wave} is the singular inverse-square potential $\sigma d_\Gamma^{ -2 }$.
Since this potential has the same scaling as $\Delta_x$, it causes fundamental difficulties in the analysis of \eqref{eq_intro_wave}.
In particular, the parameter $\sigma$, representing the strength of this singular potential, drastically alters the asymptotics of solutions $u$ at $\Gamma$.\footnote{We will justify the specific range $( -\frac{3}{4}, 0 )$ later in this section.}
On the other hand, $X$ and $V$ represent general (time-dependent) lower-order coefficients that are less singular, and the drift $X$ is allowed to contain both spatial and time components.

Our main objective is to establish a (Dirichlet) boundary observability inequality for \eqref{eq_intro_wave},
\[
\text{Energy}[u] |_{ t = 0 } \lesssim \int_{ ( -T, T ) \times \Gamma } | \text{Neumann data} [u] |^2 \text{,}
\]
which linearly estimates the energy of $u$ by its (Neumann) boundary data over a sufficiently large timespan $2 T$; see Theorem \ref{thm_main} below.
Most notably, our estimate will hold for any $n$ and hence is applicable to higher-dimensional settings.
The main ingredient for proving our observability result is a novel global Carleman estimate that is adapted to the singular potential $\sigma d_\Gamma^{-2}$; see Theorem \ref{thm_carl_est_main}.
The key assumption required for our Carleman estimate, and by extension our observability inequality, is a convexity condition on our domain $\Omega$; in particular, this condition implies $\Gamma$ is the level set of a convex function.

\subsection{Background}

Observability for wave equations is, by this point, a classical topic with decades of literature.
Thus, to maintain a concise introduction, we will restrict most of our attention near wave equations with critically singular potentials and point the reader to references within for more comprehensive surveys of observability.

In general, observability estimates have a wide variety of well-known applications.
In the present context, these lead to (linear) quantitative unique continuation, in particular implying the solution is uniquely determined by its boundary data in the observation region.
Observability also plays a fundamental role in the control theory of partial differential equations (PDEs), in that observability and controllability are known to be equivalent notions via the duality principle of \cite{dolec_russe:obs_control, MR963060}.

Many techniques have been applied to prove observability for wave equations, with some of the most successful being Fourier/spectral methods \cite{MR1366650, MR3948235, komornik_loreti:fourier_control}, microlocal methods \cite{MR1178650, MR1483711, MR3570496, MR3649373}, and multiplier/Carleman estimates \cite{MR3046295, ho:obs_wave, MR2361985, MR4797140, MR1804797, MR3964826, MR1290492, MR1786331}.
In this paper, our focus will be on Carleman estimate methods.

Due to their physical space nature, observability inequalities obtained from Carleman estimates tend to be suboptimal in terms of the size of the control region (compared to, e.g.\ the geometric control condition of the microlocal techniques).
Their primary advantage, however, is that they are applicable to an especially wide class of wave operators, in particular those containing general lower-order coefficients that can depend non-analytically on time.\footnote{In contrast, microlocal and spectral methods tend to require time-independent (or time-analytic) operators.}
This will be a key strength of our main results as well.

It is also worth mentioning that Carleman estimates have played a fundamental role in the unique continuation theory of PDEs over the past century \cite{carl:uc_strong, MR781537, tat:uc_hh}.
Moreover, just in the context of wave equations, they have found applications to a surprisingly diverse collection of problems---including control, observation, and stabilisation of PDEs; inverse problems \cite{MR630135, MR1855284}; singularity formation \cite{alex_shao:uc_nlwf}; and rigidity results in relativity \cite{alex_io_kl:unique_bh, alex_io_kl:rigid_bh, alex_schl:time_periodic, hol_shao:aads, io_kl:unique_bh}.

We now turn our attention to the singular potential $\sigma d_\Gamma^{-2}$ in \eqref{eq_intro_wave}.
As in \cite{MR4275478, MR5023025}, we refer to this as ``critical", since it possesses the same scaling as the Laplacian $\Delta_x$ and hence must be viewed and treated as a principal part of the wave operator (as opposed to the coefficients $X$ and $V d_\Gamma^{-1}$, which can be treated perturbatively).
On one hand, such singular potentials can be viewed as infinite potential wells.
Moreover, geometric extensions of \eqref{eq_intro_wave}---in which the flat background is replaced by a curved one---naturally arise from both the Klein-Gordon and Einstein equations on asymptotically anti-de Sitter settings, which are fundamental to the study of holography in theoretical physics \cite{breit_freedm:stability_sgrav, hol_shao:aads}.

The key difficulties of this article arise from this singular potential, along with the need to combine its analysis with techniques for standard wave equations.
As mentioned before, the potential completely changes the asymptotic behaviour of solutions near $\Gamma$---see the following section for details.
Moreover, since we are interested here in observability from the boundary, where the potential becomes singular, many of these issues become unavoidable here.

A vast majority of results relating to observability (and controllability) for \eqref{eq_intro_wave} apply only to $1$ spatial dimension ($n = 1$), moreover only in the simpler case of a potential that is singular on only one end of an interval:
\begin{equation}
\label{eq_sing_wave_easy} - \partial_t^2 u + \partial_x^2 u + \frac{ \sigma }{ x^2 } u = 0 \text{,} \qquad \Omega := ( 0, 1 ) \text{.}
\end{equation}
In fact, by appropriately transforming $x$ and $u$, see the computations in \cite[Appendix A]{MR3924864}, \eqref{eq_sing_wave_easy} can be transformed into a \emph{degenerate} wave equation on the same interval:
\begin{equation}
\label{eq_degen_wave_easy} - \partial_t^2 v + \partial_x ( x^\alpha \partial_x v ) = 0 \text{,} \qquad \alpha \in ( 0, 2 ) \text{.}
\end{equation}

Questions of controllability and observability for \eqref{eq_degen_wave_easy} have received considerable attention in recent years; see, e.g.\ \cite{abcl:degen_wave, bai_chai:degen_wave, bai_chai:degen_wave_int, bfm:degen_wave, fms:degen_wave, MR3227458, mlg:degen_wave, zhang_chai:degen_wave}.
A majority of results address observability from the nonsingular part of the boundary, so that some of the key difficulties here can be circumvented.
(However, we highlight \cite{MR3227458, mlg:degen_wave}, which address observability from the singular boundary.)
Also, many of these works use multiplier or Fourier methods, which are highly sensitive to the exact form of the equation, hence they do not apply to potentials that are singular at both $x = 0$ and $x = 1$, or if general lower-order terms are present in the equation.
We also mention the recent work \cite{dehoop2026boundaryobservabilitygasgiant} on some higher-dimensional analogues of \eqref{eq_degen_wave_easy}.

In light of the above, the main motivations for our boundary observability estimates are:
\begin{enumerate}
\item To obtain results for a general class of domains $\Omega$ in all dimensions.

\item To treat wave equations with general lower-order coefficients---namely, $V$ and $X$ in \eqref{eq_intro_wave}---that depend (non-analytically) on time.

\item To treat potentials $\sigma d_\Gamma^{-2}$ that become critically singular at all of $\Gamma$.
\end{enumerate}
While Carleman estimate methods would be naturally suited to tackling the aforementioned goals, developing Carleman estimates adapted to \eqref{eq_intro_wave} is known to be a very difficult problem, and very few such results exist in the literature.

Carleman estimates for (PDEs conformally equivalent to) wave equations with critically singular potentials were obtained in \cite{chatz_shao:uc_ads_gauge, hol_shao:uc_ads, hol_shao:uc_ads_ns, mcgill_shao:psc_aads}, in the context of unique continuation from the boundary of asymptotically anti-de Sitter spacetimes.
While these estimates take into account the altered boundary asymptotics of solutions in a refined way, they also only hold near the boundary, making them ill-suited for proving observability (which requires \emph{global} Carleman estimates).
Regarding the observability of \eqref{eq_intro_wave} itself, \cite{MR4275478} established the first global Carleman estimate and boundary observability for \eqref{eq_intro_wave} in higher dimensions, however the results only hold when the domain $\Omega = \mathbb{B}_1 (0)$ is a ball.\footnote{Also, for technical reasons, the result holds for all $n \geq 1$ except $n = 2$.}
%These restrictions come up due to the use of Morawetz estimates (\cite{MR234136}), which only applies to spherically symmetric domains and for $n \neq 2$.

For more general $\Omega$, key ideas arose from \cite{MR5023025}, which proved global Carleman estimates---and hence observability---for the analogous singular \emph{heat equation},
\begin{equation} \label{eq_intro_heat}
- \p_t u + \Delta_x u + \frac{\sigma}{d_\Gamma^2} u + X \cdot \nabla u + V u = 0 \text{,}
\end{equation}
now in all dimensions, but under the additional assumption that $\Gamma$ is a convex hypersurface.\footnote{This convexity assumption was removed in \cite{MR4781095}, but at a steep cost---the ensuing Carleman estimate only held locally near $\Gamma$, leading to just unique continuation rather than observability.}
In particular, \cite{MR5023025} constructed a global Carleman weight $y$ that is carefully adapted to both the distance $d_\Gamma$ and to the altered boundary asymptotics.\footnote{Earlier results for the parabolic analogue of \eqref{eq_sing_wave_easy} were treated by \cite{MR3924864} using Fourier methods; see also the recent \cite{MR5033516}. Earlier global Carleman estimates for \eqref{eq_intro_heat} in higher dimensions were developed in \cite{MR3507988}, but this applied only to internal observability and could avoid many difficulties involving boundary asymptotics.}

The main results of this paper can be viewed as analogues of those in \cite{MR5023025} for the wave equation \eqref{eq_intro_wave}.
The Carleman weight we construct is inspired by that of \cite{MR5023025}, however here we must simultaneously address issues arising from pseudoconvexity (which are crucial for wave equation Carleman estimates) that are entirely absent in parabolic settings.
Our results also solve the problem proposed in \cite[Remark 1.7]{MR4275478}---that is, we extend the results of \cite{MR4275478} to convex domains and also address the missing case $n = 2$.

\subsection{Observability}

Before stating our main observability result, we first define the appropriate quantities for analyzing \eqref{eq_intro_wave}, as well as the precise assumptions needed for $\Omega$.

\begin{definition} \label{def_deriv}
We will use the following conventions for derivatives:
\begin{itemize}
\item Let $\nabla := \nabla_x$ denote the spatial gradient on $\R^n_x$.

\item Let $\Delta := \Delta_x$ denote the Laplacian on $\R^n_x$.

\item Let $\namla := ( \partial_t, \nabla_x )$ denote the spacetime gradient on $\R_t \times \R^n_x$.

\item Let $\square := -\p_t^2 + \Delta_x$ denote the wave operator on $\R_t \times \R^n_x$.
\end{itemize}
\end{definition}

\begin{definition} \label{def_conventions}
Let $\square_\sigma$ denote the critically singular wave operator
\begin{equation*} %\label{wave_sigma}
\square_\sigma := \square + \frac{\sigma}{d_\Gamma^2} \text{.}
\end{equation*}
Moreover, let $q := q_\sigma \in ( -1, 0 )$ denote the (unique) parameter satisfying\footnote{We note that $q$ corresponds to the parameter $2 \kappa$ from \cite{MR4275478, MR5023025, MR4781095}.}
\begin{equation} \label{q_sigma}
\frac{ q ( 2 - q ) }{4} := \sigma \text{.}
\end{equation}
\end{definition}

\subsubsection{The Convexity Condition}

For our main results, the key assumption that is required for $\Omega$ is a \emph{global convexity condition}.
More precisely, this condition is characterised by the existence of \emph{a foliation for $\Omega$ whose level sets are convex and extend from $\Gamma$}:

\begin{definition} \label{def_bdf}
We say that $y \in C^4(\Omega)$ is a \emph{convex boundary defining function} (abbreviated \emph{CBDF}) for $\Omega$ if and only if the following holds:
\begin{enumerate}
\item $y > 0$ on $\Omega$.

\item There exists $0 < d_0 \ll 1$ such that $y = d_\Gamma$ on $\{ x \in \Omega \mid d_\Gamma (x) < d_0 \}$.

\item $y$ has a unique critical point $x_* \in \Omega$, with $d_\Gamma (x_*) > d_0$.

\item There exists $\gamma > 0$ such that $y$ satisfies
\begin{equation} \label{eq_angconcave}
-\xi \cdot \nabla^2 y (x) \cdot \xi \geqslant \gamma |\xi|^2
\end{equation}
for every $x \in \Omega$ and each vector $\xi \in \R^n$ tangent to the level set of $y$ at $x$.
\end{enumerate}
\end{definition}

To make sense of Definition \ref{def_bdf}, we note that (1) and (2) ensure that $y$ behaves like the boundary distance $d_\Gamma$ near $\Gamma$, and that $\Gamma$ corresponds to the level set $y = 0$.
The level sets of $y$ foliate $\Omega$ into hypersurfaces, and (4) implies that these level set are uniformly convex.
(In particular, this implies $\Gamma = \{ y = 0 \}$ is convex.)
Finally, (3) ensures this convex $y$-foliation degenerates only at a single point $x_\ast$ that is far from $\Gamma$.

One important role of a CBDF $y$ is technical in nature.
Recall that $d_\Gamma$ itself is an undesirable quantity for analysis, since it can fail to be differentiable away from $\Gamma$.
Thus, $y$ serves as an everywhere ($C^4$-)smooth substitute for $d_\Gamma$ that remains viable away from $\Gamma$.

Furthermore, $y$ will play a crucial role in our global Carleman estimate as part of the Carleman weight function.
Here, the convexity condition \eqref{eq_angconcave} is required for ensuring the weight function is pseudoconvex.
The assumption of a convex ``distance" is standard in Carleman estimate literature; see, for instance, \cite{MR2383077, MR2361985, MR1944764, MR1710233}.
However, in contrast to the above works, we only require in \eqref{eq_angconcave} convexity in directions tangent to the level sets of $y$.

Thus, the crucial assumption for $\Omega$ in our main results can be summarised as follows:

\begin{assumption}[Convexity] \label{assump_h0}
There exists a convex boundary defining function $y$ for $\Omega$.    
\end{assumption}

\begin{remark}
As a simple example, the unit ball $\Omega := \mathbb{B}_1 ( 0 )$---the setting of the observability results of \cite{MR4275478}---satisfies Assumption \ref{assump_h0}; see Proposition \ref{h0_ball}.
\end{remark}

\begin{remark}
While boundary defining functions $y$ were also used for the singular heat equation \eqref{eq_intro_heat} in \cite{MR5023025}, the convexity requirements there were less stringent, and $y$ can be constructed provided $\Gamma$ is convex.
Here, on the other hand, we require a global foliation of convex hypersurfaces on $\Omega$, and we view this as a geometric assumption for $\Omega$.
\end{remark}

\subsubsection{Boundary Asymptotics}

Well-posedness results for (geometric extensions of) \eqref{eq_intro_wave} can be found in \cite{MR3089665}, e.g.\ in energy spaces and with Dirichlet boundary conditions.
Most importantly for our purposes, \cite{MR3089665} identified the relevant notions of energy and boundary data.

First, for the energy, one can simply take the standard $H^1 \times L^2$ norm:
\begin{definition}
The energy for \eqref{eq_intro_wave} at time $\tau \in ( -T, T )$ is given by
\begin{equation} \label{eq_energy_defn}
E[u] (\tau) := \int_\Omega \left[ | \partial_t u ( \tau, x ) |^2 + | \nabla_x u ( \tau, x ) |^2 + | u ( \tau, x ) |^2 \right] \rd x \text{.}
\end{equation}
\end{definition}

In particular, the well-posedness theory implies that any solution of \eqref{eq_intro_wave} with finite energy initial data will continue to have finite energy for all times.

\begin{remark}
In fact, \cite{MR4275478, MR3089665} defined instead a ``twisted" energy, roughly with ``$\nabla_x u$" replaced by ``$y^\frac{q}{2} \nabla ( y^{ - \frac{q}{2} } u )$".
However, both the standard and twisted energies can be shown to be equivalent using the Hardy inequality of Proposition \ref{lemma_hardy_para}; see the computations in \cite{MR5023025, MR4781095}.
\end{remark}

Next, the correct notions of Dirichlet and Neumann boundary data for \eqref{eq_intro_wave}, both for well-posedness and for this article, are summarised in the following:

\begin{definition} %\label{def_bdry_DN}
For a function $u: ( -T, T ) \times \bar{\Omega} \rightarrow \R$, we define the Dirichlet and Neumann traces with respect to $\Box_\sigma$ on $\Gamma$ (assuming they exist) by
\begin{align} \label{eq_bdry_DN}
\mc{D}_q u : ( -T, T ) \times \Gamma \rightarrow \R \text{,} &\qquad \mc{D}_q u := \lim_{ d_\Gamma \searrow 0 } d_\Gamma^{-\frac{q}{2}} u \text{,} \\
\notag \mc{N}_q u : ( -T, T ) \times \Gamma \rightarrow \R \text{,} &\qquad \mc{N}_q u := - \lim_{ d_\Gamma \searrow 0 } d_\Gamma^q ( \nabla d_\Gamma \cdot \nabla ) ( d_\Gamma^{-\frac{q}{2}} u ) \text{.}
\end{align}
\end{definition}

In particular, unlike for classical wave equations, both the Dirichlet and Neumann data for solutions $u$ of \eqref{eq_intro_wave} behave like specific powers of $d_\Gamma$ (depending on $\sigma$) near $\Gamma$, and one must take this into account in order to properly extract information about $u$ on the boundary.

For our analysis, we will also need to derive some more refined boundary behaviours for Dirichlet solutions to \eqref{eq_intro_wave}.
These properties are summarised in the following definition:

\begin{definition} \label{def_bdry_ads}
A function $u: ( -T, T ) \times \bar{\Omega} \rightarrow \R$ is called \emph{boundary admissible} with respect to $\Box_\sigma$ if and only if the following conditions hold:
\begin{enumerate}
\item $\mc{N}_q u$ exists and is finite.

\item We have the following Dirichlet limits:
\begin{equation*} %\label{eq_bdry_ads} 
(1-q) \mc{D}_q ( d_\Gamma^{-1+q} u ) = - \mc{N}_q u \text{,} \qquad \mc{D}_q ( d_\Gamma^q \p_t u ) = 0 \text{.}
\end{equation*}
\end{enumerate}
\end{definition}

The conditions of Definition \ref{def_bdry_ads} capture the behaviour of ``sufficiently regular" solutions of \eqref{eq_intro_wave} with homogeneous Dirichlet data.
To justify this assertion in a relatively gentle setting, \cite{MR4275478} showed that classical Dirichlet solutions of \eqref{eq_intro_wave} with finite (appropriately defined) $H^2$-energy satisfy all the conditions in Definition \ref{def_bdry_ads}:

\begin{proposition} \label{prop_regular}
Let $u \in C^2 ((-T,T) \times \Omega)$ be a solution of \eqref{eq_intro_wave}, and assume:
\begin{itemize}
\item $\mc{D}_q u = 0$.

\item $u$ has finite (twisted) $H^2$-energy for all $\tau \in ( -T, T )$:\footnote{Again, using the Hardy inequality of Proposition \ref{lemma_hardy_para}, one can show \eqref{E2_energy} is equivalent to the $E_2$-energy defined in \cite[Section 2.4]{MR4275478}.}
\begin{equation} \label{E2_energy}
\int_\Omega | y^{ -\frac{q}{2} } \nabla [ y^q \nabla ( y^{ -\frac{q}{2} } u ) ] ( \tau, x ) |^2 dx + E [ \partial_t u ] ( \tau ) + E [u] ( \tau ) < \infty \text{.}
\end{equation}
\end{itemize}
Then $u$ is boundary admissible with respect to $\Box_\sigma$.
\end{proposition}

\begin{proof}
See \cite[Proposition 2.3]{MR4275478}.
\end{proof}

\begin{remark}
It is expected that the sharp condition needed for $u$ to be boundary admissible is that $u$ is a Dirichlet solution of \eqref{eq_intro_wave} in a specific fractional Sobolev space of order $1 - \frac{q}{2}$.
However, to avoid technicalities, we do not further refine Proposition \ref{prop_regular} in this article.
\end{remark}

\subsubsection{The Main Result}

We are now ready to state our main observability estimate:

\begin{theorem} \label{thm_main}
Assume the setting and wave equation from Assumption \ref{ass.setting}.
Furthermore, assume that $\Omega$ satisfies the convexity property of Assumption \ref{assump_h0}, and assume $T$ is sufficiently large depending on $n$, $\Omega$, and $\sigma$.

Then, for any Dirichlet solution $u$ of \eqref{eq_intro_wave} on $( -T, T ) \times \Omega$,
\begin{equation*}
\square_\sigma u + X_t \partial_t + X_x \cdot \nabla u + V u = 0 \text{,} \qquad \mc{D}_q u \equiv 0 \text{,}
\end{equation*}
that is also boundary admissible with respect to $\Box_\sigma$, we have the observability estimate
\begin{equation*}
E [u] (0) \leqslant C \int_{ (-T,T) \times \Gamma } ( \mc{N}_q u )^2 \text{,}
\end{equation*}
where the constant $C > 0$ depends on $n, T, \Omega, X, V$.
\end{theorem}

\begin{remark} \label{T_bound}
The proof of Theorem \ref{thm_main} also gives a lower bound on the observation time:
\[
T > \frac{\mu ( \sup_\Omega y )^{1+q} }{1+q} \text{.}
\]
Here, $\mu$ is a sufficiently large parameter from the corresponding Carleman estimate, Theorem \ref{thm_carl_est_main}, which depends on the size and geometry of the domain $\Omega$ (via properties of $y$).
\end{remark}

\begin{remark}
Similarly, a closer inspection of the proof of Theorem \ref{thm_main} also yields a (double exponential-type) lower bound on the observation constant.
\end{remark}

To our best knowledge, this is the first boundary observability estimate for \eqref{eq_intro_wave} for general domains and in all dimensions, at least for general time-dependent lower-order terms.
(However, we do note the recent \cite{dehoop2026boundaryobservabilitygasgiant}, which treated both observability and controllability for a general class of degenerate wave equations, though only for \emph{time-independent} operators.)

We also expect that our results can be extended to geometric settings, that is, with $\Delta$ replaced by the Laplace-Beltrami operator for a Riemannian metric.
Our methods may also extend to various degenerate wave equations in higher dimensions, which share many of the qualitative features described above---most notably, boundary traces behaving like powers of $d_\Gamma$.
However, we save these questions for future research.

On the other hand, we emphasize that Theorem \ref{thm_main} does \emph{not} imply boundary controllability for \eqref{eq_intro_wave}.
The reason is the same as in \cite{MR4275478}---the duality argument from observability to controllability also requires the opposite ``hidden regularity" estimate
\[
\int_{ (-T,T) \times \Gamma } ( \mc{N}_q u )^2 \leqslant C \cdot E [u] (0) \text{,}
\]
which is expected to be false.
In fact, the results of \cite{MR3227458} suggest observability and hidden regularity can simultaneously hold only when $E [u]$ is replaced by a Sobolev norm of fractional order $1 - \frac{q}{2} > 1$.
Thus, in order to obtain boundary control for \eqref{eq_intro_wave} via duality, one would need both observability and hidden regularity at this fractional order $1 - \frac{q}{2}$.
This would introduce significant new difficulties, which we relegate to future works.

\begin{remark}
On the other hand, the observability estimate in \cite{MR5023025} for the critically singular heat equation does lead to boundary controllability even without using sharp fractional Sobolev norms.
This is due to the additional smoothing that is available for parabolic equations.
\end{remark}

Lastly, recall we always assumed $\sigma < 0$ in Assumption \ref{ass.setting}.
It would of course be natural to consider the opposite case $\sigma \in ( 0, \frac{1}{4} )$ (corresponding to $0 < q < 1$).
However, in this case, the expected sharp Sobolev order for observability and hidden regularity would be $1 - \frac{q}{2} < 1$.
In particular, one expects Theorem \ref{thm_main} to be false for $\sigma \in ( 0, \frac{1}{4} )$, hence boundary observability for this setting is likely out of reach through the techniques of this article.

\begin{remark}
One could also consider the highly negative case $\sigma \leq -\frac{3}{4}$.
However, in this setting, the expected notion of Neumann trace would involve more than one normal derivative of $u$, leading to additional technical complications.
\end{remark}

\subsection{Carleman estimate}

The key ingredient for proving Theorem \ref{thm_main} is a novel global Carleman estimate for $\Box_\sigma$ that captures the Neumann boundary trace on $\Gamma$ and holds for any $\Omega$ satisfying Assumption \ref{assump_h0}.
Since the argument deriving observability from this Carleman estimate is standard, we focus the remainder of the introduction on the Carleman estimate itself.
Furthermore, since many ideas here are similar to those used for treating the singular heat equation in \cite{MR5023025}, we refer the reader to \cite{MR5023025} for additional insights.

In particular, our objective is to prove a weighted estimate roughly of the form
\begin{align}
\label{carleman_goal} &\int_{ ( -T, T ) \times \Omega } e^{ -2 \lambda f } | \Box_\sigma u |^2 + \int_{ ( -T, T ) \times \Gamma } e^{ -2 \lambda f } w_b | \mc{N}_q u |^2 \\
\notag &\quad \gtrsim \int_{ ( -T, T ) \times \Omega } e^{ -2 \lambda f } ( w_t | \partial_t u |^2 + w_x | \nabla_x u |^2 + w_0 u^2 ) \text{,}
\end{align}
where $e^{ -2 \lambda f }$ is the exponential Carleman weight, and where $w_b$, $w_t$, $w_x$, $w_0$ represent additional positive weights whose precise forms are not so important for the present discussion.
A crucial challenge is to find an appropriate $f$ so that \eqref{carleman_goal} holds as desired.

For our Carleman estimate, our weight function will have the form
\begin{equation} \label{intro_f}
f ( t, x ) := a - \exp h ( t, x ) \text{,} \qquad h ( t, x ) := \frac{\mu}{1+q} [ b^{1+q} - ( y (x) )^{1+q} ] - c t^2 + \beta \text{,}
\end{equation}
where $a, \beta > 0$ are fixed constants, $b := \sup_\Omega y = y ( x_\ast )$, and $\mu > 0$ is a sufficiently large constant depending on $n$, $\Omega$, $\sigma$, and $c$ is a sufficiently small constant depending only on $T$.
As is standard for physical space derivations of Carleman estimates, the main computational tasks are to conjugate the wave operator with the weight $e^{ -\lambda f }$, apply a multiplier $S v$ whose principal part is $\bar{\nabla} f$, and integrate by parts.

The main features of \eqref{intro_f} can be found in the function $h$.
First, the exponent $1 + q$ for $y$ in \eqref{intro_f} is essential; just as in \cite{MR5023025}, this specific power is needed so that after the requisite integrations by parts, one captures---courtesy of the boundary admissibility conditions of Definition \ref{def_bdry_ads}---precisely the Neumann data on the boundary.

Another crucial feature of $h$, which is exclusive to wave equation settings and is a consequence of the convexity properties \eqref{eq_angconcave} of $y$, is that the level sets of $h$ are pseudoconvex with respect to $\Box$, as defined in \cite{MR781537}.\footnote{This can be verified by an extensive computation. However, we avoid doing this in this paper, as we will not need such a precise calculation to complete our Carleman estimate.}
This is ultimately responsible for the positive first-order terms on the right-hand side of \eqref{carleman_goal}.

Notice also that our Carleman weight $e^{ - \lambda f }$ is a double exponential, with two independent large parameters $\lambda$, $\mu$.
While this is in contrast to single exponential weights used in \cite{MR4275478, MR5023025, MR4781095}, double exponential weights are a well-known tool for proving Carleman estimates; see, e.g.\ \cite{MR3046295}.
Here, the advantage of the double exponential is that it generates additional positivity for derivatives of $u$ in the $\nabla y$-direction; as a result of this, we need only assume convexity of $y$ in directions tangent to its level sets.
This gain is essential, as the boundary distance $d_\Gamma$ (which coincides with $y$ near $\Gamma$) by definition fails to be convex in non-tangent directions.
(This is in particular true for the prototypical case $\Omega := \mathbb{B}_1 (0)$.)

The above considerations suffice to produce a Carleman estimate of the form \eqref{carleman_goal}, \emph{except} near the critical point $x_\ast$ of $y$, where the zero-order weight $w_0$ can be negative.
(This is because the non-negative parts of $w_0$ contain powers of $| \nabla y |$, which vanish at $x_\ast$; see the discussions in \cite{MR5023025} for details.)
The remedy for this issue is the same as in \cite{MR5023025}; we consider \emph{two} convex boundary defining functions, $y_1 := y$ and $y_2$, with distinct critical points $x_{1,*} \neq x_{2,*}$.
(In fact, $y_2$ can be constructed by slightly perturbing $y_1$ away from $\Gamma$; see Lemma \ref{lemma_bdpair}.)

To deal with the negative terms coming from the $y_1$-estimate (near $x_{1,\ast}$), we absorb this contribution into the $y_2$-estimate.
(Since $x_{1,\ast}$ is away from $x_{2,\ast}$, the $y_2$-estimate is positive near $x_{1, \ast}$.)
Similarly, the negative part from the $y_2$-estimate (near $x_{2,\ast}$) can be absorbed into the $y_1$-estimate.
Indeed, both negative parts can be dealt with simultaneously as long as the corresponding weights $f_1$, $f_2$ are appropriately matched, via the constants $\beta_1$, $\beta_2$.\footnote{Similar tricks involving adding two Carleman estimates were also employed in \cite{alex_shao:uc_nlwf, MR4314050, MR4450884, MR4797140}.}

% Other wave Carleman estimates: \cite{MR2361985, MR3964826, MR4314050, MR4450884, MR4797140, MR5023025}

Finally, the precise statement of our global Carleman estimate can be found in Theorem \ref{thm_carl_est_main}.
While technical in nature, it is of independent interest apart from Theorem \ref{thm_main}; a key example is as a tool for solving inverse problems for wave equations; see, e.g.\ \cite{afl:calderon_curvature, MR630135, isak:ip_pde}.

\subsection{Outline}

In Section \ref{sec_prelim}, we present some preliminary results required to derive the Carleman estimate.
The heart of the analysis lies in Section \ref{sec_CE}, where we derive our Carleman estimate.
Finally, in Section \ref{sec_obs}, we apply our Carleman estimate to prove our main observability result, Theorem \ref{thm_main}.

\subsection{Acknowledgments}

The authors are heavily indebted to Bruno Vergara for sharing his preliminary notes and computations for Carleman estimates.
The authors also thank Alberto Enciso for helpful discussions.
VKJ acknowledges support from the National Board for Higher Mathematics (NBHM), Department of Atomic Energy, Government of India, under the NBHM Postdoctoral Fellowship grant number 0204/16(7)/2024/R\&D-II/6758.

\section{Preliminaries} \label{sec_prelim}

In this section, we present some properties that are needed to prove our Carleman and observability estimates.
Throughout, we adopt the setting given in Assumption \ref{ass.setting}.

We begin with a few notational conventions that will be useful in our analysis:

\begin{definition} %\label{def_domains}
We will use the following abbreviations for our space domains:
\begin{equation} \label{eq_domains}
\mc{C} := (-T,T) \times \Omega \text{,} \qquad \p \mc{C} := ( -T, T ) \times \Gamma \text{.}
\end{equation}
\end{definition}

In addition to Definition \ref{def_deriv}, it will also be useful to adopt the following notations:

\begin{definition} \label{def_geom}
Let $X$, $Y$ be (spatial) vector fields on a subset of $\R^n$, and let $\bar{X}$, $\bar{Y}$ denote (spacetime) vector fields on a subset of $\R_t \times \R^n$.
\begin{itemize}
\item Let $\eta$ and $\bar{\eta}$ denote the Euclidean and Minkowski metrics on $\R^n$ and $\R^{1+n}$, respectively, viewed as bilinear forms.
In particular, in Cartesian coordinates, we have
\[
\eta := \operatorname{diag} ( 1, \dots, 1 ) = I_{ n \times n } \text{,} \qquad \bar{\eta} := \operatorname{diag} ( -1, 1, \dots, 1 ) \text{.}
\]

\item Let $\cdot$ denote the (spatial) Euclidean inner product on $\R^n$:
\[
X \cdot Y := \eta ( X, Y ) := X_1 Y_1 + \dots + X_n Y_n \text{.}
\]

\item Let $\mdot$ denote the (spacetime) Minkowskian inner product on $\R^{1+n}$:
\[
\bar{X} \mdot \bar{Y} := \bar{\eta} ( \bar{X}, \bar{Y} ) := - \bar{X}_t \bar{Y}_t + \bar{X}_1 \bar{Y}_1 + \dots + \bar{X}_n \bar{Y}_n \text{.}
\]

\item Let $\nabla \cdot$ denote the (spatial) Euclidean divergence:
\[
\nabla \cdot X := \eta ( \nabla X ) := \partial_{ x_1 } X_1 + \dots + \partial_{ x_n } X_n \text{.}
\]

\item Let $\bar{\nabla} \mdot$ denote the (spacetime) Minkowskian divergence:
\[
\namla \mdot \bar{X} := \bar{\eta} ( \namla \bar{X} ) := - \partial_t \bar{X}_t + \partial_{ x_1 } \bar{X}_1 + \dots + \partial_{ x_n } \bar{X}_n \text{.}
\]

\item We will use $| \cdot |$ exclusively to denote the (spatial) \emph{Euclidean} norm:
\[
| X |^2 = X \cdot X \text{.}
\]
\end{itemize}
\end{definition}

Most importantly, according to Definition \ref{def_geom} above, we have the identities
\begin{align*} %\label{index_practical}
\nabla \phi \cdot \nabla \psi = \partial_{ x_1 } \phi \, \partial_{ x_1 } \psi + \dots + \partial_{ x_n } \phi \, \partial_{ x_n } \psi \text{,} &\qquad \nabla \cdot ( \nabla \phi ) = \Delta \phi \text{,} \\
\namla \phi \mdot \namla \psi = - \partial_t \phi \, \partial_t \psi + \partial_{ x_1 } \phi \, \partial_{ x_1 } \psi + \dots + \partial_{ x_n } \phi \, \partial_{ x_n } \psi \text{,} &\qquad \namla \mdot ( \namla \phi ) = \Box \phi \text{.}
\end{align*}
for sufficiently regular $\phi, \psi: ( -T, T ) \times \Omega \rightarrow \R$.
In addition, the Hessians $\nabla^2 \phi$ and $\namla^2 \phi$ can be viewed as bilinear forms, which can then be contracted with two vector fields---e.g.\ $\nabla \psi \cdot \nabla^2 \phi \cdot \nabla \psi$ or $\namla \psi \mdot \namla^2 \phi \mdot \namla \psi$---to produce scalar-valued functions.

It will also be useful to consider derivatives tangent and normal to a given CBDF:

\begin{definition} %\label{def_y_deriv}
Given a CBDF $y \in C^4 ( \Omega )$ for $\Omega$:
\begin{itemize}
\item We let $\slashed\nabla$ denote the projection of the gradient $\nabla$ or $\namla$ onto the tangent spaces of (namely, the directions tangent to) the level sets of $y$.

\item We let $D_y$ denote the following normal derivative to level sets of $y$:
\begin{equation*} %\label{eq_y_deriv}
D_y := \nabla y \cdot \nabla \text{.}
\end{equation*}

\item We let $P_{ \sigma, y }$ denote the modified singular wave operator
\begin{equation} \label{P_sigma}
P_{\sigma,y} := \Box + \frac{\sigma}{y^2} = -\p_t^2 + \Delta + \frac{\sigma}{y^2} \text{.}
\end{equation}
\end{itemize}
\end{definition}

Note that $P_{ \sigma, y }$ is a more suitable operator for analysis than $\Box_\sigma$, since $y$ remains smooth away from $\Gamma$.
In practice, one can freely replace $\Box_\sigma$ by $P_{ \sigma, y }$, since $y = d_\Gamma$ near $\Gamma$, so that our singular wave equation \eqref{eq_intro_wave} can be rewritten as
\begin{equation} \label{wave_modified}
P_{ \sigma, y } u + X \mdot \bar{\nabla} u + V_y u = 0 \text{,} \qquad V_y = V + \sigma ( d_\Gamma^{-2} - y^{-2} ) \in d_\Gamma^{-1} \cdot L^\infty ( ( -T, T ) \times \Omega; \R ) \text{.}
\end{equation}
Henceforth, we will work with $P_{ \sigma, y }$ rather than $\square_\sigma$ in our Carleman estimates.

\subsection{Boundary Defining Functions}

We now establish a few key properties of CBDFs that we will need for our main results.
First, we verify that the unit ball $\Omega := B_1 (0)$---the setting of \cite{MR4275478}---satisfies Assumption \ref{assump_h0}, hence our results indeed generalise those of \cite{MR4275478}.

\begin{proposition} \label{h0_ball}
Assumption \ref{assump_h0} holds for the unit ball $\Omega := \mathbb{B}_1 (0)$ for any $n$.
In other words, there exists a CBDF on $\mathbb{B}_1 (0)$.
\end{proposition}

\begin{proof}
Let $r = | x |$ denote the distance from the origin.
We then define $y$ as
\[
y := 1 - g (r) \text{,}
\]
where $g \in C^\infty ( 0, 1 )$ satisfies $g' > 0$, as well as
\begin{equation} \label{h0_ball_0}
g (r) := \begin{cases} r, & r_1 \leq r \leq 1 \text{,} \\ r^2, & 0 \leq r \leq r_0 \text{,} \end{cases} \qquad 0 < r_0 < r_1 < 1 \text{.}
\end{equation}
Then, by definition, $y \in C^\infty ( \mathbb{B}_1 (0) )$ is everywhere positive, and $y = 1 - r = d_\Gamma$ near $\partial \mathbb{B}_1 (0)$.
Moreover, since $g' > 0$, the only critical point of $y$ is at $x_\ast = 0$.

A direct computation yields that
\begin{equation} \label{h0_ball_1}
\nb^2 y (x) = \frac{g'(r)}{r} \, \eta + \left( g''(r) - \frac{g'(r)}{r} \right) \frac{ x \otimes x }{r^2}.
\end{equation}
Moreover, since the level sets of $y$ are simply the level sets of $r$, any $\xi \in \R^n$ that is tangent to the level set of $y$ at $x$ satisfies $\xi \cdot x = 0$. 
Thus, for such $\xi$, the formula \eqref{h0_ball_1} gives
\begin{equation} \label{h0_ball_2}
-\xi \cdot \nb^2 y (x) \cdot \xi = \frac{g'(r)}{r} |\xi|^2 \text{.}
\end{equation}
Finally, the formula \eqref{h0_ball_0} implies a positive lower bound for $ \frac{ g' (r) }{r}$,
\[
 \frac{g'(r)}{r} \geqslant \gamma > 0 \text{,}
\]
hence we conclude from \eqref{h0_ball_2} the convexity estimate \eqref{eq_angconcave}.
\end{proof}

\begin{remark}
By similar arguments, one can also show that many other convex domains $\Omega$ (i.e.\ having strictly convex $C^4$-boundaries) also satisfy Assumption \ref{assump_h0}.
\end{remark}

We also note the following general properties of CBDFs:

\begin{proposition} % \label{thm_hess_y_xx}
Let $y$ be a CBDF for $\Omega$, and assume the quantities from Definition \ref{def_bdf}.
\begin{enumerate}
\item For any $x$ with $d_\Gamma (x) < d_0$, we have
\begin{equation*} %\label{eq_nby_1}
|\nb y|^2 = 1 \text{.}
\end{equation*}

\item For all $\xi \in \R^n$, we have
\begin{equation*} %\label{eq_hess_y_xx}
- \xi \cdot \nb^2 y ( x_* ) \cdot \xi \geqslant \gamma | \xi |^2 \text{.}
\end{equation*}
In particular, on a sufficiently small neighbourhood of $x$, we have, for any $\xi \in \R^n$,
\begin{equation} \label{eq_hess_y_x*}
- \xi \cdot \nb^2 y \cdot \xi \geqslant \frac{\gamma}{2} | \xi |^2 \text{.}
\end{equation}
\end{enumerate}
\end{proposition}

Note (1) is an immediate consequence of the assumption that $y = d_\Gamma$ near $\Gamma$.
For (2), the idea is that the function $y$ will become convex in \emph{all} directions at the critical point $x_\ast$, since in the limit $x \rightarrow x_\ast$, the directions tangent to level sets of $y$ span all of $\R^n$.

Finally, we show that if $\Omega$ has a CBDF $y$, then $y$ can be perturbed into another CBDF \emph{with a different critical point}.
The ideas and proof are analogous to those in \cite{MR5023025}, though slight modifications are required due to the extra convexity requirement \eqref{eq_angconcave}.

\begin{definition} %\label{def_bdf_pair}
$(y_1,y_2)$ is a convex boundary defining pair (abbreviated CBDP) for $\Omega$ iff:
\begin{enumerate}
\item $y_1$ and $y_2$ are CBDFs for $\Omega$.

\item If $x_{1,*}$ and $x_{2,*}$ are the unique critical points of $y_1$ and $y_2$, respectively, then $x_{1,*} \neq x_{2,*}$.
\end{enumerate}
\end{definition}

\begin{lemma} \label{lemma_bdpair}
Suppose $y = y_1$ is a CBDF for $\Omega$.
Then, there exists a CBDF $y_2$ for $\Omega$ such that $(y_1, y_2)$ is a CBDP for $\Omega$.
\end{lemma}

\begin{proof}
The Morse lemma shows that there exists a neighbourhood $U \subset \{ d_\Gamma > d_0 \}$ of $x_{1,*}$ and local coordinates $z : U \rightarrow \R^n $ such that $y_1$ is a quadratic form on $U$, i.e.\ $y_1 = z \cdot A \cdot z$, with $A$ a non-singular $n \times n$ matrix.
We can also assume that $z (x_{1,*}) = 0$ and $U$ is an open ball $\mathbb{B}_{ 2 \varepsilon } (0)$ in $z$-coordinates, for small enough $\varepsilon > 0$.

Since $x_{1,*}$ is a non-degenerate critical point, by Taylor's theorem, we get that 
\begin{equation} \label{eq_bdfpair_nabla}
|\nb y_1| \geqslant c_0 > 0, \text{ on } U \setminus U',
\end{equation}
where $U' := \mathbb{B}_\varepsilon (0)$ in $z$-coordinates.
Now, let $\chi \in C^\infty(\Omega)$ be a cut-off function such that $\chi \equiv 1$ on $U'$ and $\chi \equiv 0$ on $\Omega \setminus U$. Then we define the function $y_2$ as follows
\begin{equation} \label{eq_bdfpair_delta}
y_2 := y_1 + \delta \chi b \cdot z \text{,} \qquad b \in \mathbb{S}^{n-1} \text{,} \quad \delta \ll 1 \text{.}
\end{equation}
Note $y_2 = y_1$ on $\Omega \setminus U$, hence $y_2$ satisfies the same properties as $y_1$ on the region $\{ d_\Gamma \leqslant d_0 \}$. 

For the region $\{ d_\Gamma > d_0 \}$, we see that on $U'$,\footnote{Here, $\nabla_z$ denotes the (spatial) gradient expressed with respect to the $z$-coordinates.}
\[
\nb_z y_2 = 2 A \cdot z + \delta b \text{,}
\]
which implies there is a critical point of $y_2$ at
\[
z (x_{2,*}) = - \frac{1}{2} \delta A^{-1} b \neq 0 \text{,}
\]
which is unique in $U'$ since $A$ is invertible.
Next, note that taking small enough $\delta$ in \eqref{eq_bdfpair_delta} and using \eqref{eq_bdfpair_nabla}, ensures that $|\nb_z y_2| > 0$ in $U \setminus U'$. That is, no new critical point arises in the region $U \setminus U'$. Furthermore, again using \eqref{eq_bdfpair_delta}, we see that for small enough $\delta$,
\[
- \xi \cdot \nb^2 y_2 (x) \cdot \xi \geqslant \frac{\gamma}{2} |\xi|^2 \text{,} \qquad \xi \in \R^n \text{,}
\]
on $\{ d_\Gamma > d_0 \}$.
This shows that $(y_1, y_2)$ indeed forms a boundary defining pair.
\end{proof}

\subsection{Some Estimates}

We now list a few estimates we require to prove our main results.

First, we recall the following Hardy inequality for $d_\Gamma$-weights, which plays an important role in the well-posedness theory of \eqref{eq_intro_wave}, and whose proof can be found in \cite{MR1655516}:\footnote{\cite{MR1655516} also established, in convex settings, the optimal constant for the Hardy inequality.}

\begin{proposition} \label{lemma_hardy_para}
For any $u \in H^1_0 (\Omega)$ the following holds
\[ \int_\Omega d_\Gamma^{-2} u^2 \lesssim \int_\Omega | \nabla u |^2 \text{.} \]
\end{proposition}

For our Carleman estimate, we will also need the following pointwise Hardy inequality:

\begin{lemma} \label{lemma_hardy}
For any $q \in \R$, $y \in C^2 (\Omega)$, and $v \in C^1(\Omega)$, the following inequality holds:
\begin{align}
\label{eq_hardy_anew} y^{q-1} \, ( D_y v)^2 &\geqslant \frac{(2-q)^2}{4} y^{q-3} |\nb y|^4 \, v^2 - \frac{(2-q)}{2} y^{q-2}(\Delta y |\nb y|^2 + 2 \nb y \cdot \nb^2 y \cdot \nb y ) \, v^2 \\
\notag &\qquad + \nb \cdot \left[ \frac{(2-q)}{2} y^{q-2} |\nb y|^2 \, \nb y \, v^2 \right] \text{.}
\end{align}
\end{lemma}

\begin{proof}
Let $b\in \R$, and note that
\begin{align}
\label{eq_carl_anew_-1} 0 &\leqslant ( y^{\frac{q-1}{2}} \, D_y v + b y^{\frac{q-3}{2}} |\nb y|^2 \, v )^2 \\
\notag &= y^{q-1} \, (D_y v)^2 + b^2 y^{q-3} |\nb y|^4 \, v^2 + 2b y^{q-2} |\nb y |^2 \, v D_y v \text{.}
\end{align}
In addition, a direct computation yields
\begin{align*}
\nb \cdot(b y^{q-2} |\nb y |^2 \, \nb y \, v^2) & = \nb ( b y^{q-2} |\nb y |^2 \, v^2 ) \cdot \nb y + b y^{q-2} | \nb y |^2 \Delta y \, v^2 \\
& = b (q-2) y^{q-3} |\nb y |^4 \, v^2 + 2 b y^{q-2} ( \nb y \cdot \nb^2 y \cdot \nb y ) \, v^2 \\
& \qquad + 2b y^{q-2} |\nb y |^2 \, v D_y v + b y^{q-2} | \nb y |^2 \Delta y \, v^2 \text{.}
\end{align*}
Using the above equation to replace the third term in the right hand side of \eqref{eq_carl_anew_-1}, we get
\begin{align*}
y^{q-1} \, (D_y v)^2 & \geqslant - \nb \cdot( b y^{q-2} |\nb y |^2 \, \nb y \, v^2 ) + b (q-2) y^{q-3} | \nb y |^4 \, v^2 + 2 b y^{q-2} ( \nb y \cdot \nb^2 y \cdot \nb y ) \, v^2 \\
& \qquad - b^2 y^{q-3} |\nb y |^4 \, v^2 + b y^{q-2} | \nb y |^2 \Delta y \, v^2 \text{.}
\end{align*}
Finally, substituting $b=\frac{q-2}{2}$ in the above completes the proof of the lemma.
\end{proof}

Finally, we recall the following standard energy estimate for solutions of \eqref{eq_intro_wave}:

\begin{lemma} %\label{thm_energy_est}
Assume the setting and wave equation from Assumption \ref{ass.setting}.
Then, for any Dirichlet solution $u$ of \eqref{eq_intro_wave} on $( -T, T ) \times \Omega$ that is boundary admissible with respect to $\Box_\sigma$, we have the following energy estimate,
\begin{equation} \label{eq_energy_est}
E[u] (s) \leqslant  e^{ M |t-s| } E[u] (t) \text{,} \qquad s,t \in (-T,T) \text{,}
\end{equation}
where $M$ depends on $n$, $q$, $X$, and $V$.
\end{lemma}

\begin{proof}
See \cite[Lemma 5.3]{MR4275478}.
\end{proof}

\section{Carleman Estimates} \label{sec_CE}

The objective of this section is to prove the global Carleman estimate, Theorem \ref{thm_carl_est_main}.
The general process for this is standard---we first prove a pointwise Carleman estimate, and we then integrate this over $\mc{C}$ to obtain the desired bound from Theorem \ref{thm_carl_est_main}.

\subsection{First pointwise estimate}

As usual, the first step toward our global Carleman estimate is to establish a corresponding pointwise estimate, with the eventual boundary terms represented as divergences.
Here, we establish a preliminary version of this pointwise bound, but with a zero-order weight that could be negative:

\begin{theorem} \label{thm_carl_anew}
Suppose both Assumption \ref{ass.setting} and Assumption \ref{assump_h0} hold, with $x_*$ the critical point of the CBDF $y$.
Also, fix any $0 < \delta \ll 1$, and let $q$ be as in Definition \ref{def_conventions}:
\begin{equation} \label{eq_carl_anew_0a}
-1 < q < 0 \text{,} \qquad \frac{ q ( 2 - q ) }{4} = \sigma \text{.}
\end{equation}
Then, there exist constants $C, C', C'', \mu_0 > 0$ (depending on $n$, $\Omega$, $\sigma$---but not on $T$) and $\lambda_0 > 0$ (depending only on $T$) such that the following estimate holds on $[ -T, T ] \times \bar{\Omega}$ for any $\lambda \geqslant \lambda_0$, $\mu \geqslant \mu_0$, and $u \in C^2( [-T, T] \times \bar{\Omega} )$,
\begin{align*}
\frac{1}{2} e^{-2 \lambda f} | P_{\sigma,y} u |^2 & \geqslant 2 ( \lambda \namla \mdot J ) + C \lambda e^h e^{-2\lambda f} \left[ \mu y^q | \nabla u |^2 + ( \partial_t u )^2 \right] \\
& \quad + C' \Big[ \mu^4 \lambda^3 e^{3h} y^{4q} + \mu^3 \lambda^3 e^{3 h} y^{3q-1} + \mu^2 \lambda e^h y^{2q-2} \Big] \mathbf{1}_{ \Omega \setminus \mathbb{B}_\delta (x_*) } e^{-2\lambda f} u^2 \\
& \quad - C'' \Big[ \mu^3 \lambda^3 e^{3h} y^{3q} + \mu^2 \lambda e^h y^{2q-1} + \mu \lambda e^h y^{q-2} \Big] \mathbf{1}_{\mathbb{B}_\delta(x_*)} e^{-2\lambda f} u^2 \text{,} \nt \label{eq_carl_anew_thm}
\end{align*}
where:
\begin{itemize}
\item The functions $f, h: [ -T, T ] \times \bar{\Omega} \to \R $ are given by
\begin{equation} \label{eq_carl_anew_0b}
f (t, x) := - e^{ h (t, x) } \text{,} \qquad h(t,x) := \frac{\mu}{1+q} [ b^{1+q} - (y(x))^{1+q} ] - ct^2 + \beta \text{,}
\end{equation}
with $\beta > 0$ an arbitrary constant, and with constants $b, c > 0$ satisfying
\begin{equation} \label{eq_c_choice}
b := \sup_\Omega y \text{,} \qquad c \leqslant \min \bigg( \frac{1}{T}, 1 \bigg) \text{.}
\end{equation}

\item $J := J^0 + J^1$, with the (spacetime) vector fields $J^0$, $J^1$ given by
\begin{align*}
J^1 & := \namla ( e^{-\lambda f} u ) \, ( \namla f \mdot \namla (e^{-\lambda f} u ) ) + \frac{1}{2} Q_0 \namla [ ( e^{-\lambda f} u)^2 ] \\
&\qquad - \frac{1}{2} \namla f [ \namla (e^{-\lambda f} u) \mdot \namla ( e^{-\lambda f} u ) ] - \frac{1}{2} \namla Q_0 \, ( e^{-\lambda f} u )^2 \\
&\qquad + \frac{1}{2} \namla f \, A_0 \, ( e^{-\lambda f} u )^2 \text{,} \\
\notag J^0 & := \left[ \frac{5\mu^2}{16} (1-2q) y^{2q-1} + \frac{\mu}{2} (-q)(2-q) y^{q-2} \right] e^h \, | \nabla y |^2 \, \nabla y \, ( e^{-\lambda f} u)^2 \text{.} \nt \label{eq_carl_anew_JJH}
\end{align*}

\item $Q_0, A_0: [ -T, T ] \rightarrow \R$ are given by
\begin{equation} \label{eq_carl_anew_QA}
Q_0 := \frac{1}{2} \square f + 32 e^h \text{,} \qquad A_0 := \lambda^2 ( \namla f \mdot \namla f ) + \sigma y^{-2} \text{.}
\end{equation}
\end{itemize}
\end{theorem}

\begin{remark}
Notice that the constants $\mu$ and $b$ in the Carleman estimate \eqref{eq_carl_anew_thm} can be chosen independently of $T$.
This will be important later for our main observability result, as there we must choose the timespan $T$ depending on $\mu$ and $b$.
\end{remark}

\begin{proof}
For clarity, we divide the proof into several steps.
In step A, we derive a preliminary pointwise inequality for the ``conjugated" unknown $v := e^{ -\lambda f } u$.
In steps B and C, we estimate the first-order and zero-order terms (respectively) arising from step A.
Finally, in step D, we convert the resulting bound for $v$ into our final pointwise inequality for $u$.

Before starting, we note that by \eqref{eq_c_choice}, we have
\begin{equation} \label{eq_hesty}
h(t,x) = \frac{\mu}{1+q} [ b^{1+q} - (y(x))^{1+q} ] - ct^2 + \beta \geqslant -ct^2 \geqslant -T \text{,}
\end{equation}
hence $\lambda e^h \geqslant \lambda e^{-T}$.
Thus, we simply choose $\lambda_0 := e^T$, which then ensures
\begin{equation} \label{eq_lambda_choice}
\lambda e^h \geq 1 \text{,} \qquad \lambda \geq \lambda_0 \text{.}
\end{equation}

Also, throughout the proof, we will use $C, C_1, C_2, C', C'', \dots$ to denote various positive constants.
Most importantly, we adopt the following conventions for these constants:
\begin{itemize}
\item The specific values of $C, C_1, C_2, C', C'', \dots$ can change from one line to the next.

\item The values of $C, C_1, C_2, C', C'', \dots$ can depend on $n$, $\Omega$, $\sigma$, \emph{but not on $T$}.
\end{itemize}
We also remark that the construction of $y$ depends solely on the geometry of $\Omega$.
Thus, any dependence on $b$ or on $y$ and its derivatives is viewed as dependence on $\Omega$.
Similarly, any dependence on $q$ is interpreted as dependence on $\sigma$, by Definition \ref{def_conventions}.

\vspace{0.6pc}
\noindent
\textit{Step A (preliminary inequality)}: Let $v := e^{-\lambda f} u$.
Then, by a direct computation,
\begin{align*}
e^{ -\lambda f } P_{ \sigma, y} u &= e^{-\lambda f} \square ( e^{\lambda f} v ) + \sigma y^{-2} v\\
&= e^{-\lambda f} \, \namla \mdot ( \lambda e^{\lambda f} \, \namla f \, v + e^{\lambda f} \, \namla v ) + \sigma y^{-2} v \\
&= \lambda^2 ( \namla f \mdot \namla f ) \, v + \lambda \, \square f \, v + 2 \lambda ( \namla f \mdot \namla v ) + \square v + \sigma y^{-2} v \\
&= Sv + \square v + A_0 v + \lambda \mc{E}_0 v \text{,} \nt \label{eq_carl_anew_1a}
\end{align*}
where
\[
Sv := 2 \lambda ( \namla f \mdot \namla v ) + 2 \lambda Q_0 \, v \text{,} \qquad \mc{E}_0 := \square f - 2 Q_0 \text{,}
\]
and where $Q_0$, $A_0$ are given by \eqref{eq_carl_anew_QA}.
Next, observe that
\[
e^{-\lambda f} P_{\sigma,y} u \, Sv \leqslant \frac{1}{2} e^{-2\lambda f} |P_{\sigma,y} u|^2 + \frac{1}{2} |Sv|^2 \text{.}
\]
Multiplying $Sv$ throughout \eqref{eq_carl_anew_1a} and then using the above estimate, we get
\begin{align*}
\frac{1}{2} e^{-2\lambda f} | P_{\sigma,y} u |^2 & \geqslant \frac{1}{2} |Sv|^2 + \square v S v + A_0 \, v Sv + \lambda \mc{E}_0 v Sv \\
&\geqslant \square v Sv + A_0 \, v Sv - \frac{1}{2} \lambda^2 \mc{E}_0^2 v^2 \text{,} \nt \label{eq_carl_anew_1b}
\end{align*}
where we used Young's inequality to get the last step.

Now, we expand each term on the right-hand side of \eqref{eq_carl_anew_1b}:
\begin{align*}
\square v Sv &= 2 \lambda ( \namla \mdot \namla v ) ( \namla f \mdot \namla v + Q_0 \, v ) \\
& = 2 \lambda \bigg\{ \namla \mdot \left[ \namla v \, ( \namla f \mdot \namla v ) + \frac{1}{2} Q_0 \, \namla ( v^2 ) \right] - \namla v \mdot \namla^2 f \mdot \namla v \\
& \qquad \quad - \namla f \mdot \namla^2 v \mdot \namla v - \frac{1}{2} \namla Q_0 \mdot \namla ( v^2 ) - Q_0 \, \namla v \mdot \namla v \bigg\}\\
& = 2 \lambda \bigg\{ \namla \mdot \left[ \namla v \, ( \namla f \mdot \namla v ) + \frac{1}{2} Q_0 \, \namla ( v^2 ) - \frac{1}{2} \namla f \, ( \namla v \mdot \namla v ) - \frac{1}{2} \namla Q_0 \, v^2 \right] \\
& \qquad \quad + \namla v \mdot \left[ - \namla^2 f + \frac{1}{2} ( \square f - 2 Q_0 ) \bar{\eta} \right] \mdot \namla v + \frac{1}{2} \square Q_0 \, v^2 \bigg\},\\
& = 2 \lambda \bigg\{ \namla \mdot \left[ \namla v \, ( \namla f \mdot \namla v ) + \frac{1}{2} Q_0 \, \namla ( v^2 ) - \frac{1}{2} \namla f \, ( \namla v \mdot \namla v ) - \frac{1}{2} \namla Q_0 \, v^2 \right] \\
& \qquad \quad + \namla v \mdot \left[ - \namla^2 f + \frac{1}{2} \mc{E}_0 \bar{\eta} \right] \mdot \namla v + \frac{1}{2} \square Q_0 \, v^2 \bigg\} \text{,} \nt \label{eq_carl_anew_1b1} \\
A_0 v Sv & = 2 \lambda \, A_0 v ( \namla f \mdot \namla v + Q_0 \, v ) \\
& = \lambda \, \namla \mdot ( \namla f \, A_0 \, v^2 ) - \lambda ( \namla f \mdot \namla A_0 + \square f \, A_0 ) v^2 + 2 \lambda \, A_0 Q_0 \, v^2 \\
& = \lambda \, \namla \mdot ( \namla f \, A_0 \, v^2) - \lambda ( \namla f \mdot \namla A_0 + \mc{E}_0 A_0 ) v^2 \text{.} \nt \label{eq_carl_anew_1b2}
\end{align*}
Substituting \eqref{eq_carl_anew_1b1} and \eqref{eq_carl_anew_1b2} back into \eqref{eq_carl_anew_1b} gives us
\begin{equation} \label{eq_carl_anew_2}
\frac{1}{2} e^{ -2 \lambda f } | P_{\sigma, y} u |^2 \geqslant 2 \lambda ( \namla \mdot J^1 ) + 2 \lambda \, \namla v \mdot \left( -\namla^2 f + \frac{1}{2} \mc{E}_0 \bar{\eta} \right) \mdot \namla v + A \, v^2 - \frac{1}{2} \lambda^2 \mc{E}_0^2 \, v^2 \text{,}
\end{equation}
where $J^1_\alpha$ is as in \eqref{eq_carl_anew_JJH}, and where $A$ is defined as
\begin{equation} \label{eq_carl_anew_2a}
A := \lambda ( \square Q_0 - \namla f \mdot \namla A_0 - \mc{E}_0 A_0 ) \text{.} 
\end{equation}

Before proceeding further with the proof, we establish some elementary identities that will be needed.
In particular, from \eqref{eq_carl_anew_0b} and \eqref{eq_carl_anew_QA}, we get the following:
\begin{align*} \nt \label{eq_carl_anew_derh}
\namla h & = -\mu y^q \namla y - 2ct \namla t \text{,} \\
\namla f & = - e^h \namla h \text{,} \\
\namla y \mdot \namla h & = -\mu y^q | \nabla y |^2 \text{,} \\
\p_t h & = -2ct \text{,}
\end{align*}
as well as
\begin{align*}
\namla y \mdot \namla f & = \mu e^h y^q | \nabla y |^2 \text{,} \nt \label{eq_carl_anew_2abc} \\
\namla f \mdot \namla f & = e^{2h} ( \mu^2 y^{2q} | \nabla y |^2 - 4 c^2 t^2 ) \text{,} \\
\namla^2 f & = - e^h \{ \mu^2 y^{2q} ( \namla y \otimes \namla y ) - \mu q y^{q-1} ( \namla y \otimes \namla y ) - \mu y^q \namla^2 y \\
& \qquad \qquad + 2ct \mu y^q ( \namla y \otimes \namla t + \namla t \otimes \namla y ) + ( 4 c^2 t^2 - 2c ) ( \namla t \otimes \namla t ) \} \text{,} \\
\square f & = - e^h \{ \mu^2 y^{2q} | \nb y |^2 - \mu q y^{q-1} | \nb y |^2 - \mu y^q \Delta y - ( 4 c^2 t^2 - 2c ) \} \text{,} \\
Q_0 & = e^h \left\{ - \frac{1}{2} \mu^2 y^{2q} | \nb y |^2 + \frac{1}{2} \mu q y^{q-1} | \nb y |^2 + \frac{1}{2} \mu y^q \Delta y + 2 c^2 t^2 - c + 32 \right\} \text{,} \\
\mc{E}_0 & := \square f - 2 Q_0 = - 64 e^h \text{.}
\end{align*}

\vspace{0.6pc}
\noindent
\textit{Step B (first-order terms)}: Next, we estimate the first-order terms present in \eqref{eq_carl_anew_2}.
(This process will also add in several zero-order error terms, which will be taken care of in the next step.)
For this purpose, using \eqref{eq_carl_anew_2abc}, we obtain
\begin{align*}
\namla v \mdot \left( -\namla^2 f + \frac{1}{2} \mc{E}_0 \bar{\eta} \right) \mdot \namla v & = e^h \bigg[ \mu^2 y^{2q} (D_y v)^2 - \mu q y^{q-1} (D_y v)^2 - \mu y^q ( \nabla v \cdot \nabla^2 y \cdot \nabla v ) \\
& \qquad \quad + 4ct \mu y^q D_y v \p_t v + ( 4c^2 t^2 - 2c ) ( \partial_t v )^2 - 32 \, \namla v \mdot \namla v \bigg] \\
& \geqslant e^h \bigg[ \mu^2 y^{2q} (D_y v)^2 - \mu q y^{q-1} (D_y v)^2 - \mu y^q ( \nabla v \cdot \nabla^2 y \cdot \nabla v ) \\
+ & 4 c t \mu y^q \, D_y v \p_t v - 32 | \nabla v |^2 + ( 4 c^2 t^2 - 2c + 32 ) ( \partial_t v )^2 \bigg] \text{.} \nt \label{eq_carl_anew_4}
\end{align*}
To analyze the terms on the right-hand side of \eqref{eq_carl_anew_4}, we will have to proceed differently, depending on whether one is near or away from the critical point $x_\ast$ of $y$.

First, suppose we are away from $x_*$, that is, we restrict to $\Omega \setminus \mathbb{B}_\ast$, with $\mathbb{B}_\ast$ being a small neighbourhood around $x_\ast$.
Note we then have the expansions
\begin{align}
\label{eq_carl_anew_hessian_y} | \nabla v |^2 &= \frac{ (D_y v)^2 }{ |\nb y|^2 } + | \nasla v |^2 \text{,} \\
\notag \nabla v \cdot \nabla^2 y \cdot \nabla v &= \frac{ \nabla y \cdot \nabla^2 y \cdot \nabla y }{ |\nb y|^4 } (D_y v)^2 + \frac{2}{|\nb y|^2} ( \nabla y \cdot \nabla^2 y \cdot \nasla v ) D_y v + \nasla v \cdot \nabla^2 y \cdot \nasla v \text{.}
\end{align}
Also note that on $\Omega \setminus \mathbb{B}_\ast$, we have
\begin{equation} \label{eq_carl_anew_yderbnd}
C_1 \leq |\nb y|^2 \leq C_2 \text{,} \qquad C_1, C_2 > 0 \text{.}
\end{equation}

Considering the terms on the right-hand side of \eqref{eq_carl_anew_4}, and using \eqref{eq_angconcave}, \eqref{eq_carl_anew_hessian_y}, and \eqref{eq_carl_anew_yderbnd}, we get, for some constants $C, \gamma > 0$,
\begin{align*}
& \mu^2 y^{2q} (D_y v)^2 - \mu q y^{q-1} (D_y v)^2 - \mu y^q \, ( \nabla v \cdot \nabla^2 y \cdot \nabla v ) + 4ct \mu y^q \, D_y v \p_t v - 32 | \nabla v |^2 \\
& \geqslant \mu^2 y^{2q} (D_y v)^2 - \mu q y^{q-1} (D_y v)^2 - \mu y^q \frac{ \nabla y \cdot \nabla^2 y \cdot \nabla y }{ |\nb y|^4 } (D_y v)^2 - 2 \mu y^q \, \frac{ \nabla y \cdot \nabla^2 y \cdot \nasla v }{ | \nabla v |^2 } D_y v \\
& \quad - \mu y^q ( \nasla v \cdot \nabla^2 y \cdot \nasla v ) - 32 c^2 t^2 (\p_t v)^2 - \frac{1}{8} \mu^2 y^{2q} (D_y v)^2 - 32 \frac{(D_y v)^2}{|\nb y|^2} - 32 | \nasla v |^2 \\
& \geqslant \mu^2 y^{2q} (D_y v)^2 - \mu q y^{q-1} (D_y v)^2 - \frac{C}{C_1^2} \mu y^q (D_y v)^2 - \frac{4C}{C_1^2 \gamma} \mu y^q (D_y v)^2 - \frac{\gamma}{4} \mu y^q | \nasla v |^2 \\
& \quad + \mu y^q \gamma | \nasla v |^2 - 32 c^2 t^2 (\p_t v)^2 - \frac{1}{8} \mu^2 y^{2q} (D_y v)^2 - \frac{32}{C_1} (D_y v)^2 - 32 | \nasla v |^2 \\
& \geqslant \frac{ 3 \mu^2 }{4} y^{2q} (D_y v)^2 - \mu q y^{q-1} (D_y v)^2 + \frac{\mu}{2} \gamma y^q | \slashed \nb v|^2 - 32 c^2 t^2 (\p_t v)^2 \text{,} \nt \label{eq_carl_anew_Z1}
\end{align*}
for large enough $\mu$.
(Notice here that $C_1, C_2, C, \gamma$---and hence $\mu$---can all be chosen independently of $T$.)
Using \eqref{eq_carl_anew_Z1} in \eqref{eq_carl_anew_4} we see that the following holds on $\Omega \setminus \mathbb{B}_\ast$:
\begin{align*}
& \namla v \mdot \left( - \namla^2 f + \frac{1}{2} \mc{E}_0 \bar{\eta} \right) \mdot \namla v \\
& \geqslant e^h \bigg[ \left( \frac{3\mu^2}{4} y^{2q} - \mu q y^{q-1} \right) (D_y v)^2 + \frac{\mu}{2} \gamma y^q | \nasla v |^2 + ( - 2 c + 32 - 28 c^2 t^2) (\partial_t v)^2 \bigg] \\
& \geqslant e^h \bigg[ \left( \frac{5\mu^2}{8} y^{2q} - \mu q y^{q-1} \right) (D_y v)^2 + \frac{\mu}{4} \gamma y^q | \nabla v|^2 + ( - 2 c + 32 - 28 c^2 t^2) (\partial_t v)^2 \bigg] \text{.} \nt \label{eq_carl_anew_6a}
\end{align*}

Next, on $\mathbb{B}_\ast$ itself (that is, near the critical point $x_*$), we use that \eqref{eq_hess_y_x*} implies
\[
- \nabla v \cdot \nabla^2 y \cdot \nabla v \geqslant \frac{\gamma}{2} | \nabla v |^2 \text{,}
\]
which implies that, for large enough $\mu$,
\[
- \mu y^q ( \nabla v \cdot \nabla^2 y \cdot \nabla v ) - 32 | \nb v |^2 \geqslant \frac{ \mu \gamma }{4} y^q \, | \nb v |^2 \text{.}
\]
Thus, we obtain the following on $\mathbb{B}_\ast$:
\begin{align*}
& \mu^2 y^{2q} (D_y v)^2 - \mu q y^{q-1} (D_y v)^2 - \mu y^q \, ( \nabla v \cdot \nabla^2 y \cdot \nabla v ) + 4ct \mu y^q \, D_y v \p_t v - 32 | \nabla v |^2 \\
& \geqslant \mu^2 y^{2q} (D_y v)^2 - \mu q y^{q-1} (D_y v)^2 + \frac{\mu \gamma}{4} y^q | \nabla v |^2 - \frac{1}{8} \mu^2 y^{2q} (D_y v)^2 - 32 c^2 t^2 (\p_t v)^2 \\
& \geqslant \left( \frac{5 \mu^2}{8} y^{2q} - \mu q y^{q-1} \right) (D_y v)^2 + \frac{\mu \gamma}{4} y^q | \nabla v |^2 - 32 c^2 t^2 (\p_t v)^2 \text{.}
\end{align*}
The above discussion and \eqref{eq_carl_anew_4} imply that on $\mathbb{B}_\ast$, we have 
\begin{align*}
& \namla v \mdot \left( - \namla^2 f + \frac{1}{2} \mc{E}_0 \bar{\eta} \right) \mdot \namla v \nt \label{eq_carl_anew_6a2}\\
& \geqslant e^h \bigg[ \left( \frac{5 \mu^2}{8} y^{2q} - \mu q y^{q-1} \right) (D_y v)^2 + \frac{\mu \gamma}{4} y^q | \nabla v |^2 + ( - 2 c + 32 - 28 c^2 t^2 ) (\partial_t v)^2 \bigg] \text{.}
\end{align*}

Combining \eqref{eq_carl_anew_6a}--\eqref{eq_carl_anew_6a2}, we now obtain, everywhere on $\Omega$, the inequality
\begin{align*}
& \namla v \mdot \left( - \namla^2 f + \frac{1}{2} \mc{E}_0 \bar{\eta} \right) \mdot \namla v \nt \label{eq_carl_anew_6a3}\\
& \geqslant e^h \bigg[ \left( \frac{5 \mu^2}{8} y^{2q} - \mu q y^{q-1} \right) (D_y v)^2 + \frac{\mu \gamma}{4} y^q | \nabla v |^2 + ( - 2 c + 32 - 28 c^2 t^2 ) (\partial_t v)^2 \bigg] \text{.}
\end{align*}
For the coefficient of $(\p_t v)^2$ in the above, note that \eqref{eq_c_choice} implies
\begin{align*}
- 2 c + 32 - 28 c^2 t^2 &\geqslant - 2 c + 32 - 28 c^2 T^2 \\
&\geqslant -2 + 32 - 28 \\
&> 1 \text{.}
\end{align*}
Then, using \eqref{eq_carl_anew_6a3} and the above in \eqref{eq_carl_anew_2}, and recalling \eqref{eq_lambda_choice} and \eqref{eq_carl_anew_2abc}, we have
\begin{align*}
\frac{1}{2} e^{-2\lambda f} | P_{\sigma,y} u |^2 & \geqslant 2 \lambda ( \namla \mdot J^1 ) + 2 \lambda e^h \bigg[ \left( \frac{5 \mu^2}{8} y^{2q} - \mu q y^{q-1} \right) (D_y v)^2 + \frac{\mu \gamma}{4} y^q | \nabla v |^2 + (\partial_t v)^2 \bigg] \\
& \qquad + \left[ A - \frac{1}{2} 64^2 \lambda^2 e^{2h} \right] v^2 \\
& \geqslant 2 \lambda ( \namla \mdot J^1 ) + 2 \lambda e^h \bigg[ \left( \frac{5 \mu^2}{8} y^{2q} - \mu q y^{q-1} \right) (D_y v)^2 + \frac{\mu \gamma}{4} y^q | \nabla v |^2 + (\partial_t v)^2 \bigg] \\
& \qquad + ( A - C'' \mu^3 \lambda^3 e^{3h} y^{3q} ) v^2 \text{,} \nt \label{eq_carl_anew_6} 
\end{align*}
where $C''$ here and further below are independent of $T$.

Next, we use the pointwise Hardy inequality Lemma \ref{lemma_hardy} for the terms in \eqref{eq_carl_anew_6} corresponding to $( D_y v )^2$.
For the $y^{q-1} (D_y v)^2 $ term we directly use \eqref{eq_hardy_anew}; for the $y^{2q} (D_y v)^2$ term we modify \eqref{eq_hardy_anew} using the transformation $q \rightarrow 2q+1$.
From this and \eqref{eq_carl_anew_JJH}, we obtain
\begin{align*}
\frac{5\mu^2}{8} & y^{2q} (D_y v)^2 - \mu q y^{q-1} (D_y v)^2 \\
& \geqslant \frac{5\mu^2}{32} (1-2q)^2 y^{2q-2} |\nb y|^4 \, v^2 + \frac{\mu}{4} (-q) (2-q)^2 y^{q-3} |\nb y|^4 \, v^2 \\
& \qquad - \left\{ \frac{5\mu^2}{16} (1-2q) y^{2q-1} + \frac{\mu}{2} (-q)(2-q) y^{q-2} \right\} ( \Delta y |\nb y|^2 + 2 \nb y \cdot \nb^2 y \cdot \nb y ) \, v^2 \\
& \qquad + \nabla \cdot \left[ \left\{ \frac{5\mu^2}{16} (1-2q) y^{2q-1} + \frac{\mu}{2} (-q)(2-q) y^{q-2} \right\} |\nb y|^2 \, \nb y \, v^2 \right] \\
& \geqslant \nabla \cdot ( e^{-h} J_0 ) + \frac{5\mu^2}{32} (1-2q)^2 y^{2q-2} |\nb y|^4 \, v^2 + \frac{\mu}{4} (-q) (2-q)^2 y^{q-3} |\nb y|^4 \, v^2 \\
& \qquad - C'' ( \mu^2 y^{2q - 1 } + \mu y^{q-2} ) \, v^2 \text{.} \nt \label{eq_carl_anew_6y}
\end{align*}

For the last term in the above expression, we see that
\[
2 \lambda e^h \, \namla \mdot (e^{-h} J^0) = 2 \lambda ( \nabla \cdot J^0 - \namla h \mdot J^0 ) \text{.} \nt \label{eq_carl_anew_divJ_extra}
\]
For the second term in the above equation, using \eqref{eq_carl_anew_derh} shows that
\begin{align*}
-2 & \lambda \namla h \mdot J^0 \\
& = -2 \lambda \left\{ \frac{5\mu^2}{16} (1-2q) y^{2q-1} + \frac{\mu}{2} (-q)(2-q) y^{q-2} \right\} e^h |\nb y|^2 [ ( -\mu y^q \namla y - 2ct \namla t ) \mdot \namla y ] \, v^2 \\
& = \lambda e^h \left\{ \frac{5\mu^3}{8} (1-2q) y^{3q-1} + \mu^2 (-q)(2-q) y^{2q-2} \right\} |\nb y|^4 \, v^2 \text{.}
\end{align*}
Recalling \eqref{eq_carl_anew_JJH}, we substitute \eqref{eq_carl_anew_6y}, \eqref{eq_carl_anew_divJ_extra}, and the above into \eqref{eq_carl_anew_6} to get
\begin{align*}
\frac{1}{2} e^{-2 \lambda f} |P_{\sigma,y} u|^2 & \geqslant 2 \lambda ( \namla \mdot J ) + 2 \lambda e^h \left[ \frac{\mu \gamma}{4} y^q |\nabla v|^2 + (\partial_t v)^2 \right] \\
& \qquad + \bigg[ \lambda e^h \frac{5\mu^2}{16} (1-2q)^2 y^{2q-2} |\nb y|^4 + \lambda e^h \frac{\mu}{2} (-q) (2-q)^2 y^{q-3} |\nb y|^4 \\
& \qquad \quad + \lambda e^h \frac{5\mu^3}{8} (1-2q) y^{3q-1} |\nb y|^4 + \lambda e^h \mu^2 (-q)(2-q) y^{2q-2} |\nb y|^4 \\
& \qquad \quad + A - C'' ( \mu^2 \lambda e^h y^{2q - 1 } + \mu \lambda e^h y^{q-2} - \mu^3 \lambda^3 e^{3h} y^{3q} ) \bigg] v^2 \\
& = 2 \lambda ( \namla \mdot J ) + 2 \lambda e^h \left[ \frac{\mu \gamma}{4} y^q |\nabla v|^2 + (\partial_t v)^2 \right] + I v^2 \text{,} \nt \label{eq_carl_anew_7}
\end{align*}
where $I$ denotes the coefficient of $v^2$.

\vspace{0.6pc}
\noindent
\textit{Step C (zero-order terms)}: Next, we expand the term $I v^2$ in \eqref{eq_carl_anew_7}, in particular $A v^2$.
While this is a rather extensive computation, the process can be simplified by observing that:
\begin{itemize}
\item The relatively few ``leading-order terms" have the same coefficients as in the second line of \eqref{eq_carl_anew_thm}---namely, $\mu^4 \lambda^3 e^{3h} y^{4q}$, $\mu^3 \lambda^3 e^{3h} y^{3q - 1}$, and $\mu^2 \lambda e^h y^{2q - 2}$.

\item The remaining ``error terms" can be controlled by the coefficients in the third line of \eqref{eq_carl_anew_thm}---namely, $\mu^3 \lambda^3 e^{3h} y^{3q}$, $\mu^2 \lambda e^h y^{2q - 1}$, and $\mu \lambda e^h y^{q - 2}$.
\end{itemize}
In particular, we only require the precise forms of the leading-order terms.

Recalling \eqref{eq_carl_anew_QA}, we compute the second and third terms of $A$ in \eqref{eq_carl_anew_2a} as follows:
\begin{align*}
- \lambda ( \namla f \mdot \namla A_0 ) & = \lambda e^h ( \namla h \mdot \namla A_0 ) \\
& = \lambda e^h ( -\mu y^q \nb^\alpha y \nb_\alpha A_0 + 2ct \p_t A_0 ) \\
%& = - e^h \Big[ \mu^4 2 \lambda^2 e^{2h} y^{4q} |\nb y|^4 + \mu^3 \{ -2q \lambda^2 e^{2h} y^{3q-1} |\nb y|^4 - \lambda^2 e^{2h} y^{3q} \nb y \cdot \nb^2 y \cdot \nb y \} \\
%& \qquad + \mu^2 \{ -8 \lambda^2 e^{2h} c^2 t^2 y^{2q} |\nb y|^2 - 8 \lambda^2 e^{2h} c^2 t^2 |\nb y|^2 \} + 2 \mu \sigma y^{q-3} |\nb y|^2 \\
%& \qquad + 32 \lambda^2 e^{2h} c^4 t^4 - 16 \lambda^2 e^{2h} c^3 t^2 \Big] \\
& \geqslant \lambda e^h ( 2 \mu^4 \lambda^2 e^{2h} y^{4q} |\nb y|^4 - 2 \mu^3 q \lambda^2 e^{2h} y^{3q-1} |\nb y|^4 + 2 \mu \sigma y^{q-3} |\nb y|^2 ) \\
& \qquad - C'' ( \mu^3 \lambda^3 e^{3h} y^{3q} ) \text{,} \nt \label{eq_carl_anew_Y1} \\
- \lambda \mc{E}_0 A_0 & = 64 \lambda e^h [ \lambda^2 e^{2h} (\mu^2 y^{2q} |\nb y|^2 - 4 c^2 t^2 ) + \sigma y^{-2} ] \\
%& = - 64 e^h \big[ \mu^2 \lambda^2 e^{2h} y^{2q} |\nb y|^2 - 4 c^2 t^2 \lambda^2 e^{2h} + \sigma y^{-2} \} \big] \\
& \geqslant - C'' ( \mu^2 \lambda^3 e^{3h} y^{2q} + \mu \lambda e^h y^{q - 2} ) \text{.} \nt \label{eq_carl_anew_Y2}
\end{align*}
For $\namla f \mdot \namla A_0$, the leading terms are obtained when the maximum number of derivatives hit powers of $y$ and $e^h$, yielding the largest powers of $y^{-1}$ and $\mu$.
Moreover, for the error terms, we also used the following immediate consequence of \eqref{eq_c_choice},
\begin{equation} \label{c_bound}
c \leq 1 \text{,} \qquad c|t| \leq 1 \text{,}
\end{equation}
and we can freely introduce additional powers of $\mu$, $y^{-1}$, and $\lambda e^h$---the last using \eqref{eq_lambda_choice}.
In particular, thanks to \eqref{c_bound}, the constants $C''$ remain independent of $T$.

Finally, letting $Q_0 := e^h B$, where (by \eqref{eq_carl_anew_2abc})
\begin{equation} \label{eq_carl_anew_Bdef}
B := - \frac{1}{2} \mu^2 y^{2q} |\nb y|^2 + \frac{1}{2} \mu q y^{q-1} |\nb y|^2 + \frac{1}{2} \mu y^q \Delta y + 2 c^2 t^2 - c + 32 \text{,}
\end{equation}
we can write
\begin{align*}
\lambda \square Q_0 & = \lambda [ -\p_t^2 (e^h B) + \Delta (e^h B) ] \\
& = \lambda e^h B [ -(\p_t h)^2 + |\nb h|^2 + \square h ] + \lambda e^h [ -2 \p_t h \p_t B + 2 \nb h \cdot \nb B + \square B ] \\
& =: T_1 + T_2. \nt \label{eq_carl_anew_Y90}
\end{align*}
For the above terms, we expand using \eqref{eq_c_choice}, \eqref{eq_carl_anew_derh}, and \eqref{eq_carl_anew_Bdef} to obtain
\begin{align*}
T_1 & = \lambda e^h \left( - \frac{1}{2} \mu^2 y^{2q} |\nb y|^2 + \frac{1}{2} \mu q y^{q-1} |\nb y|^2 + \frac{1}{2} \mu y^q \Delta y + 2c^2t^2 - c + 32 \right) \\
& \qquad \quad \cdot ( \mu^2 y^{2q} |\nb y|^2 - \mu q y^{q-1} |\nb y|^2 - \mu y^q \Delta y -4c^2 t^2 + 2c ) \\
& \geqslant \lambda e^h \left( -\frac{1}{2} \mu^4 y^{4q} |\nb y|^4 + \mu^3 q y^{3q-1} |\nb y|^4 - \frac{1}{2} \mu^2 q^2 y^{2q-2} |\nb y|^4 \right) \\
& \qquad - C'' ( \mu^3 \lambda^3 e^{3h} y^{3q} + \mu^2 \lambda e^h y^{2q-1} ) \text{,} \nt \label{eq_carl_anew_Y9a} \\
T_2 & \geqslant \lambda e^h \left[ 2 \mu^3 q y^{3q-1} |\nb y|^4 - \mu^2 q (3q-2) y^{2q-2} |\nb y|^4 + \frac{1}{2} \mu q (q-1) (q-2) y^{q-3} |\nb y|^4 \right] \\
& \qquad - C'' ( \mu^3 \lambda^3 e^{3h} y^{3q} + \mu^2 \lambda e^h y^{2q-1} + \mu \lambda e^h y^{q-2} ) \text{.} \nt \label{eq_carl_anew_Y9b}
\end{align*}
Though lengthy, this can be made systematic by again noting that the leading terms arise when the maximum number of derivatives hit powers of $y$ and $e^h$.
Another key observation is that each instance of $t$ in $T_1$ and $T_2$ is multiplied by $c$; using that $c |t| \leq 1$ via \eqref{c_bound}, we see that $C''$ remains independent of $T$.
Substituting \eqref{eq_carl_anew_Y9a} and \eqref{eq_carl_anew_Y9b} into \eqref{eq_carl_anew_Y90} yields
\begin{align*}
\lambda \square Q_0 & \geqslant \lambda e^h \bigg[ \mathord{-}\frac{1}{2} \mu^4 y^{4q} |\nb y|^4 + \mu^3 q y^{3q-1} |\nb y|^4 - \frac{1}{2} \mu^2 q^2 y^{2q-2} |\nb y|^4 + 2 \mu^3 q y^{3q-1} |\nb y|^4 \\
& \qquad \quad - \mu^2 q (3q-2) y^{2q-2} |\nb y|^4 + \frac{1}{2} \mu q (q-1) (q-2) y^{q-3} |\nb y|^4 \bigg] \\
& \qquad - C'' ( \mu^3 \lambda^3 e^{3h} y^{3q} + \mu^2 \lambda e^h y^{2q-1} + \mu \lambda e^h y^{q-2} ) \text{.}
\end{align*}

Adding \eqref{eq_carl_anew_Y1}, \eqref{eq_carl_anew_Y2}, and the above yields an inequality for $A$, by \eqref{eq_carl_anew_2a}.
From this, we see that the coefficient $I$ in \eqref{eq_carl_anew_7} satisfies
\begin{align*}
I & \geqslant \lambda e^h \bigg[ 2 \mu^4 \lambda^2 e^{2h} y^{4q} |\nb y|^4 + 2 \mu^3 (-q) \lambda^2 e^{2h} y^{3q-1} |\nb y|^4 + 2 \mu \sigma y^{q-3} |\nb y|^2 - \frac{1}{2} \mu^4 y^{4q} |\nb y|^4 \\
& \qquad \quad + \mu^3 q y^{3q-1} |\nb y|^4 - \frac{1}{2} \mu^2 q^2 y^{2q-2} |\nb y|^4 + 2 \mu^3 q y^{3q-1} |\nb y|^4 - \mu^2 q (3q-2) y^{2q-2} |\nb y|^4 \\
& \qquad \quad + \frac{1}{2} \mu q (q-1) (q-2) y^{q-3} |\nb y|^4 + \frac{5\mu^2}{16} (1-2q)^2 y^{2q-2} |\nb y|^4 \\
& \qquad \quad + \frac{\mu}{2} (-q) (2-q)^2 y^{q-3} |\nb y|^4 + \frac{5\mu^3}{8} (1-2q) y^{3q-1} |\nb y|^4 + \mu^2 (-q)(2-q) y^{2q-2} |\nb y|^4 \bigg] \\
& \qquad - C'' ( \mu^3 \lambda^3 e^{3h} y^{3q} + \mu^2 \lambda e^h y^{2q-1} + \mu \lambda e^h y^{q-2} ) \\
& = \lambda e^h \bigg[ \mu^4 y^{4q} |\nb y|^4 \left\{ 2 \lambda^2 e^{2h} - \frac{1}{2} \right\} + \mu^3 y^{3q-1} |\nb y|^4 \left\{ 2 (-q) \lambda^2 e^{2h} + 3q + \frac{5}{8} (1-2q) \right\} \\
& \qquad \quad + \mu^2 y^{2q-2} |\nb y|^4 \left\{ -\frac{q^2}{2} - q(3q-2) + \frac{5}{16} (1-2q)^2 + (-q) (2-q)  \right\} \\
& \qquad \quad + \mu y^{q-3} |\nb y|^4 \left\{ 2 \sigma |\nb y|^{-2} + \frac{1}{2} q (q-1) (q-2) + \frac{1}{2} (-q) (2-q)^2 \right\} \bigg] \\
& \qquad - C'' ( \mu^3 \lambda^3 e^{3h} y^{3q} + \mu^2 \lambda e^h y^{2q-1} + \mu \lambda e^h y^{q-2} ) \\
& =: I_4 + I_3 + I_2 + I_1 - C'' ( \mu^3 \lambda^3 e^{3h} y^{3q} + \mu^2 \lambda e^h y^{2q-1} + \mu \lambda e^h y^{q-2} ) \text{,} \nt \label{eq_carl_anew_Y5}
\end{align*}
where $I_j$ consists of the terms with $\mu^j$, for $j = 1,2,3,4$.

For $I_4$, the choice \eqref{eq_lambda_choice} ensures $2 \lambda^2 e^{2h} - \frac{1}{2} > 1$, hence we obtain
\begin{equation} \label{eq_carl_anew_Z2}
I_4 \geqslant \mu^4 \lambda^3 e^{3h} y^{4q} |\nabla y|^4 \text{.}
\end{equation}

For $I_3$, the choice \eqref{eq_lambda_choice} similarly ensures that
\begin{align*}
I_3 &\geqslant \mu^3 \lambda^3 e^{3 h} y^{3q-1} |\nb y|^4 \left\{ 2 (-q) + 3q + \frac{5}{8} (1-2q) \right\} \\
&\geqslant \frac{5}{8} \mu^3 \lambda^3 e^{3 h} y^{3q-1} |\nb y|^4 \text{.} \nt \label{eq_carl_anew_7d}
\end{align*}

For $I_2$, we use \eqref{eq_carl_anew_0a} and directly compute
\begin{align*}
I_2 & = \lambda e^{h} \mu^2 y^{2q-2} |\nabla y|^4 \left\{ -\frac{5}{4} q^2 - \frac{5}{4} q + \frac{5}{16} \right\} \\
& \geqslant \frac{5}{16} \mu^2 \lambda e^h |\nabla y|^4 y^{2q-2} \text{.} \nt \label{eq_carl_anew_8}
\end{align*}

For $I_1$, we recall \eqref{eq_carl_anew_0a} and note that
\begin{align*}
\frac{1}{2} & \{ q (q-1) (q-2) - q(2-q)^2\} y^{q-3} |\nb y|^4 + 2 \sigma y^{q-3} |\nb y|^2 \\
& = \left[ - \frac{1}{2} q (2-q) + 2 \sigma - 2 \sigma + 2 \sigma |\nb y|^{-2} \right] y^{q-3} |\nb y|^4\\
& = \left[ - \frac{q}{2}(2-q) + 2 \sigma \right] y^{q-3} |\nb y|^4 + 2 \sigma (|\nb y|^{-2} -1) y^{q-3} |\nb y|^4\\
& = 2 \sigma (|\nb y|^{-2} -1) y^{q-3} |\nb y|^4 \text{,}
\end{align*}
and that this term vanishes near $\Gamma$, since $|\nb y|^2 = 1$ there.
Away from $\Gamma$, we have $y^{-1} \lesssim 1$, hence we are free to include additional powers of $y$.
As a result, we obtain an error bound:
\begin{equation} \label{eq_carl_anew_Z4} 
I_1 \geqslant - C'' \mu \lambda e^h y^{q-2} \text{.}
\end{equation}

Combining \eqref{eq_carl_anew_Y5}-\eqref{eq_carl_anew_Z4}, and substituting into \eqref{eq_carl_anew_7}, the above discussion shows that 
\begin{align*}
\frac{1}{2} e^{-2 \lambda f} |P_{\sigma,y} u|^2 & \geqslant 2 \lambda ( \namla \mdot J ) + 2 \lambda e^h \left[ \frac{\mu \gamma}{4} y^q |\nabla v|^2 + (\partial_t v)^2 \right] \\
& \quad + C' \Big[ \mu^4 \lambda^3 e^{3h} y^{4q} + \mu^3 \lambda^3 e^{3 h} y^{3q-1} + \mu^2 \lambda e^h y^{2q-2} \Big] |\nabla y|^4 \, v^2 \\
& \quad - C'' \Big[ \mu^3 \lambda^3 e^{3h} y^{3q} + \mu^2 \lambda e^h y^{2q-1} + \mu \lambda e^h y^{q-2} \Big] v^2 \text{.} % \nt \label{eq_carl_anew_9}
\end{align*}
Now, taking $\delta > 0$ small enough, we have that $| \nabla y | \gtrsim 1$ on $x \in \Omega \setminus \mathbb{B}_\delta(x_*)$, i.e.\ away from the only critical point of $y$.
Then, after adjusting $C'$, the above estimate reduces to
\begin{align*}
\frac{1}{2} e^{-2 \lambda f} |P_{\sigma,y} u|^2 & \geqslant 2 \lambda ( \namla \mdot J ) + 2 \lambda e^h \left[ \frac{\mu \gamma}{4} y^q |\nabla v|^2 + (\partial_t v)^2 \right] \\
& \quad + C' \Big[ \mu^4 \lambda^3 e^{3h} y^{4q} + \mu^3 \lambda^3 e^{3 h} y^{3q-1} + \mu^2 \lambda e^h y^{2q-2} \Big] \mathbf{1}_{\Omega \setminus \mathbb{B}_\delta(x_*)} v^2\\
& \quad - C'' \Big[ \mu^3 \lambda^3 e^{3h} y^{3q} + \mu^2 \lambda e^h y^{2q-1} + \mu \lambda e^h y^{q-2} \Big] v^2 \text{.} %\nt \label{eq_carl_anew_10}
\end{align*}
Moreover, on $\Omega \setminus \mathbb{B}_\delta(x_*)$, the negative term in the right-hand side of the above estimate can be absorbed into the positive term by taking $\mu$ to be large enough.
(Since $C''$ is independent of $T$, then $\mu$ remains independent of $T$ as well.)
As a result, we have
\begin{align*}
\frac{1}{2} e^{-2 \lambda f} |P_{\sigma,y} u|^2 & \geqslant 2 \lambda ( \namla \mdot J ) + 2 \lambda e^h \left[ \frac{\mu \gamma}{4} y^q |\nabla v|^2 + (\p_t v)^2 \right] \\
& \quad + C' \Big[ \mu^4 \lambda^3 e^{3h} y^{4q} + \mu^3 \lambda^3 e^{3 h} y^{3q-1} + \mu^2 \lambda e^h y^{2q-2} \Big] \mathbf{1}_{\Omega \setminus \mathbb{B}_\delta(x_*)} v^2\\
& \quad - C'' \Big[ \mu^3 \lambda^3 e^{3h} y^{3q} + \mu^2 \lambda e^h y^{2q-1} + \mu \lambda e^h y^{q-2} \Big] \mathbf{1}_{\mathbb{B}_\delta(x_*)} v^2 \text{.} \nt \label{eq_carl_anew_11}
\end{align*}

\vspace{0.6\baselineskip}
\noindent \textit{Step D (reversal of conjugation)}: Recalling that $v = e^{-\lambda f} u$, we note that
\[
\namla u = e^{\lambda f} \, \namla v + \lambda e^{\lambda f} \namla f \, v \text{.}
\]
Then we see that
\begin{align*} 
e^{-2\lambda f} |\nabla u|^2 &\leqslant 2 |\nabla v|^2 + 2 \lambda^2 |\nabla f|^2 v^2 \\
&= 2 |\nabla v|^2 + 2 \mu^2 \lambda^2 e^{2h} y^{2q} |\nabla y|^2 e^{-2\lambda f} u^2 \text{.} \nt \label{eq_carl_anew_12}
\end{align*}
Similarly, recalling also \eqref{eq_c_choice}, we have
\begin{align*}
e^{-2\lambda f} (\p_t u)^2 &\leqslant 2 (\p_t v)^2 + 2 \lambda^2 (\p_t f)^2 v^2 \\
&\leqslant 2 (\p_t v)^2 + 8 \lambda^2 e^{2h} e^{-2\lambda f} u^2 \text{.} \nt \label{eq_carl_anew_13}
\end{align*}
Finally, we multiply \eqref{eq_carl_anew_12} and \eqref{eq_carl_anew_13} by $\lambda e^h$, then substitute into \eqref{eq_carl_anew_11}.
On $\Omega \setminus \mathbb{B}_\delta ( x_\ast )$, the negative zero-order terms can again be absorbed into the positive terms in \eqref{eq_carl_anew_11} by taking large $\mu$.
This yields \eqref{eq_carl_anew_thm} and concludes the proof of the theorem.
\end{proof}

\subsection{The Second Pointwise Estimate}

Note that there are a few negative terms present on the right hand side of pointwise Carleman estimate \eqref{eq_carl_anew_thm}.
To eliminate these negative terms, we will apply \eqref{eq_carl_anew_thm} about two different CBDFs and then sum their contributions.
This is carried out in the following theorem and proof:

\begin{theorem} \label{lemma_2pts}
Suppose both Assumptions \ref{ass.setting} and \ref{assump_h0} hold, and let $q$ be as in \eqref{eq_carl_anew_0a}.
Then, there exists a CBDP $( y_1 := y, \, y_2 )$, constants $C, C', \mu_0 > 0$ (depending on $n$, $\Omega$, $\sigma$---but not on $T$), and $\lambda_0 > 0$ (depending only on $T$), such that the following estimate holds on $[ -T, T ] \times \bar{\Omega}$ for any $\lambda \geqslant \lambda_0$, $\mu \geqslant \mu_0$ and $u \in C^2( [-T, T] \times \bar{\Omega} )$, 
\begin{align*}
\frac{1}{2} e^{-2 \lambda f_j} |P_{\sigma,y_j} u|^2 & \geqslant 2 \lambda \sum_{j=1}^2 ( \namla \mdot J_j ) + C \lambda \sum_{j=1}^2 e^{h_j} e^{-2\lambda f_j} \left[ \mu y_j^q |\nabla u|^2 + (\p_t u)^2 \right] \\
& \quad + C' \sum_{j=1}^2 \Big[ \mu^4 \lambda^3 e^{3h_j} y_j^{4q} + \mu^3 \lambda^3 e^{3 h_j} y_j^{3q-1} + \mu^2 \lambda e^{h_j} y_j^{2q-2} \Big] e^{-2\lambda f_j} u^2 \text{,} \nt \label{eq_carl_anew_2pt_0}
\end{align*}
where:
\begin{itemize}
\item The functions $f_j, h_j: [ -T, T ] \times \bar{\Omega} \to \R $, where $j = 1, 2$, are given by
\begin{equation} \label{eq_carl_anew_2pt_0b}
f_j (t,x) := - e^{ h_j (t, x) }, \quad h_j (t,x) := \frac{\mu}{1+q} [b_j^{1+q} - (y_j(x))^{1+q}] - ct^2 + \beta_j \text{,}
\end{equation}
with $b_j := \sup_\Omega y_j$, with $\beta_j > 0$ appropriately chosen, and with $c$ satisfying \eqref{eq_c_choice}.

\item $J_j := J_j^0 + J_j^1$, for $j = 1, 2$, with (spacetime) vector fields $J_j^0$, $J_j^1$ given by
\begin{align*} 
J^1_j & := \namla ( e^{-\lambda f_j} u ) \, ( \namla f_j \mdot \namla (e^{-\lambda f_j} u ) ) + \frac{1}{2} Q_{0,j} \namla [ ( e^{-\lambda f_j} u)^2 ] \\
&\qquad - \frac{1}{2} \namla f_j [ \namla_j (e^{-\lambda f_j} u) \mdot \namla ( e^{-\lambda f_j} u ) ] - \frac{1}{2} \namla Q_{0,j} \, ( e^{-\lambda f_j} u )^2 \\
&\qquad + \frac{1}{2} \namla f_j \, A_{0,j} \, ( e^{-\lambda f_j} u )^2 \text{,} \\
\notag J^0_j & := \left[ \frac{5\mu^2}{16} (1-2q) y_j^{2q-1} + \frac{\mu}{2} (-q)(2-q) y_j^{q-2} \right] e^h_j \, | \nabla y_j |^2 \, \nabla y_j \, ( e^{-\lambda f_j} u)^2 \text{.} \nt \label{eq_carl_anew_JJH_2pt}
\end{align*}

\item $Q_{0,j}, A_{0,j}: [ -T, T ] \rightarrow \R$, for $j = 1, 2$, are given by
\begin{equation} \label{eq_carl_anew_2pt_Q}
Q_{0,j} := \frac{1}{2} \square f_j + 32 e^{h_j} \text{,} \qquad A_{0,j} := \lambda^2 ( \namla f_j \mdot \namla f_j ) + \sigma y_j^{-2} \text{.}
\end{equation}
\end{itemize}
\end{theorem}

\begin{proof}
We first apply Lemma \ref{lemma_bdpair} to the assumed CBDF $y$ to obtain a CBDP $(y_1 = y, \, y_2)$.
In particular, $y_j$ has a unique critical point $x_j \in \Omega$ satisfying $d_\Gamma(x_j) > d_0$ for each $j = 1, 2$, and $x_1 \neq x_2$.
Note in addition that the maximum attained by $y_j$ is $b_j := y_j(x_j)$.
As a result, there exist constants $0 < r_1 < b_1$, $0 < r_2 < b_2$, and $\delta > 0$ satisfying
\begin{align*}
\mathbb{B}_\delta(x_1) \cap \mathbb{B}_\delta(x_2) & = \emptyset, \\
\{ r_1 \leqslant y_1 \leqslant b_1 \} \cap \{ y_2 \leqslant r_2 \} & \supseteq \mathbb{B}_\delta(x_1),  \nt \label{eq_jcarl_anew_1}\\
\{ r_2 \leqslant y_2 \leqslant b_2 \} \cap \{ y_1 \leqslant r_1 \} & \supseteq \mathbb{B}_\delta(x_2).
\end{align*}

We now apply Theorem \ref{thm_carl_anew} twice, with $y := y_1$ and $y := y_2$, with $\delta$ as above, and with values of the constants $\beta_1$, $\beta_2$ to be determined below.
Adding the two estimates yields
\begin{align*}
\sum_{j=1}^2 \frac{1}{2} e^{-2\lambda f_j} & |P_{\sigma,y_j} u|^2 \geqslant 2 \lambda \sum_{j=1}^2 ( \namla \mdot J_j ) + C \lambda \sum_{j=1}^2 e^{h_j} e^{-2\lambda f_j} \left[ \mu y_j^q |\nabla u|^2 + (\p_t u)^2 \right] \\
& \qquad + C' \sum_{j=1}^2 \Big[ \mu^4 \lambda^3 e^{3h_j} y_j^{4q} + \mu^3 \lambda^3 e^{3 h_j} y_j^{3q-1} + \mu^2 \lambda e^{h_j} y_j^{2q-2} \Big] \textbf{1}_{ \Omega \setminus \mathbb{B}_{\delta} (x_j) } e^{-2\lambda f_j} u^2 \\
& \qquad - C'' \sum_{j=1}^2 \Big[ \mu^3 \lambda^3 e^{3 h_j} y_j^{3q} + \mu^2 \lambda e^{h_j} y_j^{2q-1} + \mu \lambda e^{h_j} y_j^{q-2} \Big] \mathbf{1}_{ \mathbb{B}_\delta (x_j)} e^{-2\lambda f_j} u^2. \nt \label{eq_2pt_anew_1} 
\end{align*}
(In particular, we can ensure both estimates have the same set of constants $\mu$, $C$, $C'$, $C''$, $\lambda$, and with $\mu$, $C$, $C'$, $C''$ being independent of $T$.)

Next, for $j = 1, 2$, let us define $\bar{f}_j, \bar{h}_j: [ -T, T ] \rightarrow \R$ by
\[
\bar{f}_j (t) := - e^{ \bar{h}_j (t) } \text{,} \qquad \bar{h}_j (t) := \frac{\mu}{1+q} ( b_j^{1+q} - r_j^{1+q} ) - ct^2 + \beta_j \text{.}
\]
We now choose $\beta_1, \beta_2 > 0$ to satisfy
\[
\beta_1 - \beta_2 = \frac{\mu}{1+q} \left( b_2^{1+q} - b_1^{1+q} - r_2^{1+q} + r_1^{1+q} \right) \text{,}
\]
so that we have
\begin{equation} \label{eq_jcarl_anew_2}
\bar h_1(t) = \bar h_2(t) \text{,} \qquad \bar f_1(t) = \bar f_2(t).
\end{equation}

Let $j^* := 3-j$, for $j = 1, 2$.
Then, on $\mathbb{B}_\delta (x_j)$, we have
\begin{align*}
h_j (t,x) \leqslant \bar h_j (t) \text{,} &\qquad - f_j (t,x) \leqslant - \bar f_j (t) \text{,} \\ \nt \label{eq_2pt_X}
\bar h_{j^*}(t) \leqslant h_{j^*}(t,x) \text{,} &\qquad - \bar f_{j^*}(t) \leqslant - f_{j^*}(t,x) \text{.}
\end{align*}
Furthermore, there is some $C > 0$ (depending only on $\Omega$ and $q$) such that
\begin{align*}
\mu^3 \lambda^3 e^{3h_j(t,x)} e^{-2\lambda f_j(t,x)} y_j^{3q} \textbf{1}_{ \mathbb{B}_{\delta} (x_j) } & \leqslant \mu^3 \lambda^3 e^{3 \bar h_j(t)} e^{-2\lambda \bar f_j(t)} y_j^{3q} \textbf{1}_{ \mathbb{B}_{\delta} (x_j)}\\
& \leqslant C \mu^3 \lambda^3 e^{3 \bar h_{j^*}(t)} e^{-2\lambda \bar f_{j^*}(t)} y_{j^*}^{4q} y_{j^*}^{-4q} \textbf{1}_{ \mathbb{B}_{\delta} (x_j) }\\
& \leqslant C \mu^3 \lambda^3 e^{3h_{j^*}(t,x)} e^{-2\lambda f_{j^*}(t,x)} y_{j^*}^{4q} \textbf{1}_{ \mathbb{B}_{\delta} (x_j)}\\
& \leqslant C \mu^3 \lambda^3 e^{3h_{j^*}(t,x)} e^{-2\lambda f_{j^*}(t,x)} y_{j^*}^{4q} \textbf{1}_{ \Omega \setminus \mathbb{B}_{\delta} (x_{j^*})}, \nt \label{eq_2pt_2}
\end{align*}
where the first step used the first line of \eqref{eq_2pt_X}, the second step used \eqref{eq_jcarl_anew_2} and that $y_j, y_j^\ast \simeq 1$ on $\mathbb{B}_\delta (x_j)$, the third step used the last line of \eqref{eq_2pt_X}, and the last step is due to set inclusion. 
We can similarly treat the remaining negative terms in \eqref{eq_2pt_anew_1}:
\begin{align*}
\mu^2 \lambda e^{h_j(t,x)} e^{-2\lambda f_j(t,x)} y_j^{2q-1} \textbf{1}_{ \mathbb{B}_{\delta} (x_j) } & \leqslant C \mu^2 \lambda e^{h_{j^*}(t,x)} e^{-2\lambda f_{j^*}(t,x)} y_{j^*}^{3q-1} \textbf{1}_{ \Omega \setminus \mathbb{B}_{\delta} (x_{j^*})} \text{,} \\
\mu \lambda e^{h_j(t,x)} e^{-2\lambda f_j(t,x)} y_j^{q-2} \textbf{1}_{ \mathbb{B}_{\delta} (x_j) } & \leqslant C \mu \lambda e^{h_{j^*}(t,x)} e^{-2\lambda f_{j^*}(t,x)} y_{j^*}^{2q-2} \textbf{1}_{ \Omega \setminus \mathbb{B}_{\delta} (x_{j^*})} \text{.} \nt \label{eq_2pt_3}
\end{align*}

Thus, from \eqref{eq_2pt_2}--\eqref{eq_2pt_3}, we see that, by taking $\mu$ sufficiently large and recalling that $\lambda e^{ h_{ j } } \geq 1$ (from the choice $\lambda \geq e^T$ in the proof of Theorem \ref{thm_carl_anew}), the first negative term on the last line of \eqref{eq_2pt_anew_1} can be absorbed into the first positive term in the second line of \eqref{eq_2pt_anew_1}, and similarly for the second and third respective negative and positive terms.
Thus, we have
\begin{align*}
\sum_{j=1}^2 \frac{1}{2} e^{-2\lambda f_j} |P_{\sigma,y_j} u|^2 & \geqslant 2 \lambda \sum_{j=1}^2 ( \namla \mdot J_j ) + C \lambda \sum_{j=1}^2 e^{h_j} e^{-2\lambda f_j} \left[ \mu y_j^q |\nabla_x u|^2 + (\p_t u)^2 \right] \\
& + C' \sum_{j=1}^2 \Big[ \mu^4 \lambda^3 e^{3h_j} y_j^{4q} + \mu^3 \lambda^3 e^{3 h_j} y_j^{3q-1} + \mu^2 \lambda e^{h_j} y_j^{2q-2} \Big] \textbf{1}_{ \Omega \setminus \mathbb{B}_{\delta} (x_j) } e^{-2\lambda f_j} u^2.
\end{align*}
Finally, notice we can remove $\textbf{1}_{ \Omega \setminus \mathbb{B}_{\delta} (x_j) }$ in the above estimate, since $\Omega \setminus \mathbb{B}_\delta ( x_1 )$ and $\Omega \setminus \mathbb{B}_\delta ( x_2 )$ cover all of $\Omega$, and since, similar to \eqref{eq_2pt_2}--\eqref{eq_2pt_3}, the inequalities \eqref{eq_2pt_X} also imply
\begin{align*}
\mu^4 \lambda^3 e^{3 h_j(t,x)} e^{-2\lambda f_j(t,x)} y_j^{4q} \textbf{1}_{ \mathbb{B}_{\delta} (x_j) } & \leqslant \mu^4 \lambda^3 e^{3 h_{j^*}(t,x)} e^{-2\lambda f_{j^*}(t,x)} y_{j^*}^{4q} \textbf{1}_{ \Omega \setminus \mathbb{B}_{\delta} (x_{j^*})} \text{,} \\
\mu^3 \lambda^3 e^{3 h_j(t,x)} e^{-2\lambda f_j(t,x)} y_j^{3q-1} \textbf{1}_{ \mathbb{B}_{\delta} (x_j) } & \leqslant \mu^3 \lambda^3 e^{3 h_{j^*}(t,x)} e^{-2\lambda f_{j^*}(t,x)} y_{j^*}^{3q-1} \textbf{1}_{ \Omega \setminus \mathbb{B}_{\delta} (x_{j^*})} \text{,} \\
\mu^2 \lambda e^{h_j(t,x)} e^{-2\lambda f_j(t,x)} y_j^{2q-2} \textbf{1}_{ \mathbb{B}_{\delta} (x_j) } & \leqslant \mu^2 \lambda e^{h_{j^*}(t,x)} e^{-2\lambda f_{j^*}(t,x)} y_{j^*}^{2q-2} \textbf{1}_{ \Omega \setminus \mathbb{B}_{\delta} (x_{j^*})} \text{.}
\end{align*}
This completes the proof of the \eqref{eq_carl_anew_2pt_0}.
\end{proof}

\subsection{The Integrated Estimate}

Finally, we state and prove our main Carleman estimate.
This is a consequence of integrating our combined pointwise estimate from Theorem \ref{lemma_2pts}.

\begin{theorem} \label{thm_carl_est_main}
Suppose both Assumptions \ref{ass.setting} and \ref{assump_h0} hold, and let $q$ be as in \eqref{eq_carl_anew_0a}.
Then, there exists a CBDP $( y_1 := y, \, y_2 )$, constants $C, C', C_b, \mu_0 > 0$ (depending on $n$, $\Omega$, $\sigma$---but not on $T$), and $\lambda_0 > 0$ (depending only on $T$), such that the following holds on $[ -T, T ] \times \bar{\Omega}$ for any $\lambda \geqslant \lambda_0$, $\mu \geqslant \mu_0$, and for any boundary admissible $u \in C^2( [-T, T] \times \bar{\Omega} )$ with respect to $\square_\sigma$ that is supported on $\mathcal{C} \cap \{ |t| < T - \delta \}$ for some $0 < \delta \ll T$,
\begin{align*}
&\sum_{j=1}^2 \int_{\mc{C}} e^{-2\lambda f_j} |P_{\sigma,y_j} u|^2 + C_b \mu \lambda \sum_{j=1}^2 \int_{\p\mc{C}} e^{h_j} e^{-2 \lambda f_j} (\mc{N}_q u)^2 \\
& \quad \geqslant C \lambda \sum_{j=1}^2 \int_\mc{C} e^{h_j} e^{-2\lambda f_j} [ \mu y_j^q |\nabla u|^2 + (\p_t u)^2] \\
& \quad \qquad + C' \sum_{j=1}^2 \int_\mc{C} e^{-2\lambda f_j} ( \mu^4 \lambda^3 e^{3 h_j} y_j^{4q} + \mu^3 \lambda^3 e^{3 h_j} y_j^{3q-1} + \mu^2 \lambda e^{h_j} y_j^{2q-2} ) u^2 \text{,}
\end{align*}
where $f_j, h_j: [ -T, T ] \times \bar{\Omega}$, $j = 1, 2$, are as defined in Theorem \ref{lemma_2pts},
\[
f_j (t, x) := - e^{ h_j (t, x) }, \quad h_j (t, x) := \frac{\mu}{1+q} [ b_j^{1+q} - (y_j(x))^{1+q} ] - ct^2 + \beta_j \text{,}
\]
with $b_j$, $c$, $\beta_j$ chosen as in the statement of Theorem \ref{lemma_2pts}.
\end{theorem}

\begin{proof}
Since the assumptions of Theorem \ref{lemma_2pts} are satisfied, there exist a CBDP $( y_1, y_2 )$ and constants $C$, $C'$, $\mu_0$, $\lambda_0$ as in the statement of Theorem \ref{lemma_2pts}.
Now, fix any $\lambda \geq \lambda_0$, $\mu \geq \mu_0$, and boundary admissible $u$ as in the above theorem statement.
Theorem \ref{lemma_2pts} now implies that \eqref{eq_carl_anew_2pt_0} holds for our given $\lambda$, $\mu$, $u$, with $f_j$, $h_j$, $J_j = J_j^0 + J_j^1$, $Q_{ 0, j }$, $A_{ 0, j }$ (for $j = 1, 2$) defined as in \eqref{eq_carl_anew_2pt_0b}--\eqref{eq_carl_anew_2pt_Q}.
Fix now $0 < \epsilon \ll 1$, and note the regions $\{ y_1 > \epsilon \}$ and $\{ y_2 > \epsilon \}$ are the same because $y_1 = y_2 = d_\Gamma$ near $\Gamma$.

We now integrate \eqref{eq_carl_anew_2pt_0} on the domain $\mc{C}_\epsilon := (-T,T) \times \{ y_j > \epsilon \}$ to obtain
\begin{align*}
\sum_{j=1}^2 \int_{ \mc{C}_\epsilon } \frac{1}{2} e^{-2\lambda f_j} & |P_{\sigma,y_j} u|^2 \geqslant 2 \lambda \sum_{j=1}^2 \int_{\mc{C}_\epsilon} ( \namla \mdot J_j ) + C \lambda \sum_{j=1}^2 \int_{\mc{C}_\epsilon} e^{h_j} e^{-2\lambda f_j} [ \mu y_j^q |\nabla u|^2 + (\p_t u)^2] \\
& + C' \sum_{j=1}^2 \int_{\mc{C}_\epsilon} e^{-2\lambda f_j} ( \mu^4 \lambda^3 e^{3 h_j} y_j^{4q} + \mu^3 \lambda^3 e^{3 h_j} y_j^{3q-1} + \mu^2 \lambda e^{h_j} y_j^{2q-2} ) u^2 \text{.} \nt \label{eq_carl_est_int_0a} 
\end{align*}
Letting $\p\mc{C}_\epsilon := (-T,T) \times \{ y_j = \epsilon \}$ and applying the divergence theorem to the first term on the right-hand side of \eqref{eq_carl_est_int_0a}, we have that
\begin{align*}
\sum_{j=1}^2 \int_{ \mc{C}_\epsilon } \frac{1}{2} e^{-2\lambda f_j} & |P_{\sigma,y_j} u|^2 + 2 \lambda \sum_{j=1}^2 \int_{\p\mc{C}_\epsilon} ( \nabla y_j \cdot J_j ) \geqslant C \lambda \sum_{j=1}^2 \int_{\mc{C}_\epsilon} e^{h_j} e^{-2\lambda f_j} [ \mu y_j^q |\nabla u|^2 + (\p_t u)^2] \\
& + C' \sum_{j=1}^2 \int_{\mc{C}_\epsilon} e^{-2\lambda f_j} ( \mu^4 \lambda^3 e^{3 h_j} y_j^{4q} + \mu^3 \lambda^3 e^{3 h_j} y_j^{3q-1} + \mu^2 \lambda e^{h_j} y_j^{2q-2} ) u^2 \text{.} \nt \label{eq_carl_div_bdry} 
\end{align*}

We now estimate the boundary term in \eqref{eq_carl_div_bdry}, for fixed $j = 1, 2$.
For this purpose, note that by Definition \ref{def_bdry_ads}, for $v := e^{ -\lambda f_j} u$ and any continuous $g : [-T,T] \times \bar{\Omega} \to \R$, we have
\begin{align*}
\lim_{\epsilon \searrow 0} \int_{\p\mc{C}_\epsilon} g(t,y_j) y_j^q (\p_t v)^2 & = 0 \text{,} \\
\lim_{\epsilon \searrow 0} \int_{\p\mc{C}_\epsilon} g(t,y_j) y_j^q (D_{y_j} v)^2 & = \frac{(2-q)^2}{4(1-q)^2} \int_{\p\mc{C}} g(t,y_j) e^{-2 \lambda f_j} (\mc{N}_q u)^2 \text{,} \nt \label{eq_carl_est_int_1} \\
\lim_{\epsilon \searrow 0} \int_{\p\mc{C}_\epsilon} g(t,y_j) y_j^{q-2} v^2 & = \frac{1}{(1-q)^2} \int_{\p\mc{C}} g(t,y_j) e^{-2 \lambda f_j} (\mc{N}_q u)^2 \text{.}
\end{align*}
In particular, we want to estimate the following terms at the boundary:
\begin{align*}
\nabla y_j \cdot J_j & = ( \namla f_j \mdot \namla v ) \, D_{y_j} v + \frac{1}{2} Q_{0,j} \, D_{y_j} (v^2) - \frac{1}{2} ( \namla v \mdot \namla v ) \, D_{y_j} f_j - \frac{1}{2} D_{y_j} Q_{0,j} \, v^2 \\
& \qquad + \frac{1}{2} D_{y_j} f_j A_{0,j} \, v^2 + \left\{ \frac{5\mu^2}{16} (1-2q) {y_j}^{2q-1} + \frac{\mu}{2} (-q)(2-q) {y_j}^{q-2} \right\} e^{h_j} |\nb {y_j}|^4 \, v^2 \\
& = \mu e^{h_j} y_j^q \, (D_{y_j} v)^2 + 2 e^{h_j} ct \, \p_t v D_{y_j} v + \frac{1}{2} Q_{0,j} \, D_{y_j} (v^2) \\
& \quad - \frac{1}{2} \mu e^{h_j} y_j^q |\nb y_j|^2 \, [ |\nb v|^2 - (\p_t v)^2 ] - \frac{1}{2} D_{y_j} Q_{0,j} \, v^2 + \frac{1}{2} \mu e^{h_j} y_j^q |\nb y_j|^2 A_{0,j} \, v^2 \\
& \quad + \left\{ \frac{5\mu^2}{16} (1-2q) {y_j}^{2q-1} + \frac{\mu}{2} (-q)(2-q) {y_j}^{q-2} \right\} e^{h_j} |\nb y_j|^4 \, v^2 \\
& =: I_1 + I_2 + \dots + I_7 \text{.} \nt \label{eq_carl_est_int_1b}
\end{align*}

$I_1$ can be controlled using the second equation from \eqref{eq_carl_est_int_1},
\begin{equation} \label{eq_carl_999}
\lim_{\epsilon \searrow 0} \int_{\p\mc{C}_\epsilon} I_1 \leqslant C'' \mu \int_{\p\mc{C}} e^{h_j} e^{-2 \lambda f_j} (\mc{N}_q u)^2 \text{,}
\end{equation}
with $C'' > 0$ a constant that depends on $n$, $\Omega$, $q$ but not on $T$.
(We now also let $C''$ change values between lines.)
For $I_2$, we proceed as follows and apply \eqref{eq_carl_est_int_1}:
\begin{equation} \label{eq_carl_est_int_2a}
\left| \lim_{\epsilon \searrow 0} \int_{\p\mc{C}_\epsilon} I_2 \right| \leqslant C'' \lim_{\epsilon \searrow 0} \int_{\p\mc{C}_\epsilon} e^{h_j} y_j^q (\p_t v)^2 + C'' \lim_{\epsilon \searrow 0} \int_{\p\mc{C}_\epsilon} e^{h_j} y_j^{-q} (D_{y_j} v)^2 = 0 \text{.}
\end{equation}
Next, using \eqref{eq_carl_anew_2pt_Q} and \eqref{eq_carl_est_int_1}, we estimate $I_3$ as follows:
\begin{equation}
\left| \lim_{\epsilon \searrow 0} \int_{\p\mc{C}_\epsilon} I_3 \right| \leqslant C'' \mu \int_{\p\mc{C}} e^{h_j} e^{-2 \lambda f_j} (\mc{N}_q u)^2 \text{.} \label{eq_carl_est_int_2b}
\end{equation}
(In particular, only the second term in the equation for $Q_0$ in \eqref{eq_carl_anew_2abc} makes a non-trivial contribution to the boundary limit.)
Now, for $I_4$, note only the second term (containing $( \partial_t u )^2$) needs to be estimated, since the first term (with $| \nabla u |^2$) has good sign:
\begin{align*}
\lim_{\epsilon \searrow 0} \int_{\p\mc{C}_\epsilon} I_4 \leqslant \frac{1}{2} \lim_{\epsilon \searrow 0} \int_{\p\mc{C}_\epsilon} \mu e^{h_j} y_j^q (\p_t v)^2 = 0 \text{.}
\end{align*}

The remaining three terms in \eqref{eq_carl_est_int_1b} can be treated similarly:
\begin{align*}
\left| \lim_{\epsilon \searrow 0} \int_{\p\mc{C}_\epsilon} I_5 \right| & \leqslant C'' \lim_{\epsilon \searrow 0} \int_{\p\mc{C}_\epsilon} \mu e^{h_j} y_j^{q-2} e^{-2 \lambda f_j} u^2 \leqslant C'' \mu \int_{\p\mc{C}} e^{h_j} e^{-2 \lambda f_j} (\mc{N}_q u)^2 \text{,} \\
\left| \lim_{\epsilon \searrow 0} \int_{\p\mc{C}_\epsilon} I_6 \right| & \leqslant C'' \lim_{\epsilon \searrow 0} \int_{\p\mc{C}_\epsilon} \mu e^{h_j} y_j^{q-2} v^2 \leqslant C'' \mu \int_{\p\mc{C}} e^{h_j} e^{-2 \lambda f_j} (\mc{N}_q u)^2 \text{,} \\
\left| \lim_{\epsilon \searrow 0} \int_{\p\mc{C}_\epsilon} I_7 \right| & \leqslant C'' \mu \int_{\p\mc{C}} e^{h_j} e^{-2 \lambda f_j} (\mc{N}_q u)^2 \text{.} \nt \label{eq_carl_est_int_2c}
\end{align*}
(Note for $I_6$, only the last term in the formula for $A_{ 0, j }$ in\eqref{eq_carl_anew_2pt_Q} contributes to the limit.)

Finally, combining \eqref{eq_carl_est_int_1b}--\eqref{eq_carl_est_int_2c} yields that
\[
\int_{\mc{C}} ( \nabla y_j \cdot J_j ) \leqslant C'' \mu \int_{\p\mc{C}} e^{h_j} e^{-2 \lambda f_j} (\mc{N}_q u)^2 \text{.}
\]
Using the above estimate in \eqref{eq_carl_div_bdry} completes the proof of the theorem.
\end{proof}

\section{Observability} \label{sec_obs}

In this final section, we prove our main observability result, Theorem \ref{thm_main}.

\subsection{Proof of Theorem \ref{thm_main}}

Since Assumptions \ref{ass.setting} and \ref{assump_h0} hold, we can apply Theorem \ref{thm_carl_est_main}, which yields a CBDP $( y_1, y_2 )$, constants $C, C', C_b, \mu_0 > 0$ (depending on $n$, $\Omega$, $\sigma$ but not on $T$) and $\lambda_0 > 0$ (depending only on $T$).
We also let $f_j$ and $h_j$ (where $j = 1, 2$) be defined as in Theorem \ref{thm_carl_est_main}; we also set $b := \max ( b_1, b_2 )$ and $c := \min ( \frac{1}{T}, 1 )$.

Now, let $u$ be any boundary admissible (with respect to $\square_\sigma$) Dirichlet solution of \eqref{eq_intro_wave} on $( -T, T ) \times \Omega$, and fix $\mu \geqslant \mu_0$ and $\lambda \geqslant \lambda_0$, whose values will be determined later.

We then choose $T$ and $\mu$ large enough such that
\begin{equation} \label{T_choice}
T > \frac{\mu b^{1+q}}{1+q} \geqslant 1 \text{,} \qquad c T^2 = T \text{.}
\end{equation}
Then, from \eqref{eq_carl_anew_2pt_0b}, we see that the following holds for $j = 1, 2$:
\begin{align*}
\sup_{\mc{C} \cap \{t=\pm T\}} ( -f_j ) &= \sup_{ \bar{\Omega} } \exp \left[ \frac{\mu (b_j^{1+q}-y_j^{1+q})}{1+q} - c T^2 + \beta_j \right] < e^{\beta_j} \text{,} \\
\inf_{\mc{C} \cap \{t=0\}} ( -f_j ) &= \inf_{ \bar{\Omega} } \exp \left[ \frac{\mu (b_j^{1+q}-y_j^{1+q})}{1+q} + \beta_j \right] \geqslant e^{\beta_j} \text{.}
\end{align*}
Thus, there exist constants $0 < \delta \ll T$ and $\alpha_j < \beta_j$ such that, for $j=1,2$,
\begin{equation} \label{eq_obs_pf_1}
\begin{rcases}
h_j \leqslant \alpha_j \text{,} \quad \mathord{-} f_j \leqslant e^{\alpha_j} & \qquad \text{for } t \in (-T, -T + \delta) \cup (T-\delta,T) \text{,} \\
h_j \geqslant \alpha_j \text{,} \quad \mathord{-} f_j \geqslant e^{\alpha_j} & \qquad \text{for } t \in (-\delta, \delta) \text{.}
\end{rcases}
\end{equation}

We define the regions $I_\delta$, $J_\delta$ by
\begin{equation*} %\label{eq_obs_pf_2}
I_\delta := [-T+\delta, T-\delta] \text{,} \qquad J_\delta := (-T,-T+\delta) \cup (T-\delta,T) \text{.}
\end{equation*}
We also let $\psi \in C^\infty (\bar{\mc{C}})$ be a cut-off function such that
\begin{itemize}
\item $\psi$ depends only on $t$.

\item $\psi ( t, x ) = 1$, for $t \in I_\delta$.

\item $\psi ( t, x ) = 0$, for $t = [ -T, -T + \frac{\delta}{2} ] \cup [ T - \frac{\delta}{2}, T ]$.
\end{itemize}
Applying the Carleman estimate of Theorem \ref{thm_carl_est_main} to $\psi u$ then yields
\begin{align*}
& \sum_{j=1}^2 \int_{\mc{C}} e^{-2\lambda f_j} |P_{\sigma,y_j} (\psi u)|^2 + C_b \mu \lambda \sum_{j=1}^2 \int_{\p\mc{C}} e^{h_j} e^{-2 \lambda f_j} \psi^2 (\mc{N}_q u)^2 \\
& \quad \geqslant C \lambda \sum_{j=1}^2 \int_\mc{C} e^{h_j} e^{-2\lambda f_j} [ \mu y_j^q \psi^2 |\nabla u|^2 + |\p_t (\psi u)|^2 ] \\
& \quad \qquad + C' \sum_{j=1}^2 \int_\mc{C} e^{-2\lambda f_j} ( \mu^4 \lambda^3 e^{3 h_j} y_j^{4q} + \mu^3 \lambda^3 e^{3 h_j} y_j^{3q-1} + \mu^2 \lambda e^{h_j} y_j^{2q-2} ) \psi^2 u^2 \\
& \quad \geqslant C'' \sum_{j=1}^2 \int_{I_\delta \times \Omega} \lambda e^{h_j} e^{-2\lambda f_j} [ | \nabla u |^2 + (\p_t u)^2 + u^2 ] \text{,} \nt \label{eq_obs_pf_3}
\end{align*}
since $\mu$ can be chosen large, $\lambda e^{ h_j } \geq 1$, the $y_j$'s are bounded above, and $\psi \equiv 1$ on $I_\delta \times \mc{C}$.

For the left-hand side of \eqref{eq_obs_pf_3}, using \eqref{P_sigma}, \eqref{wave_modified}, and the definition of $\psi$, we obtain
\begin{align*}
|P_{\sigma,y_j} (\psi u)| & \leqslant |\psi P_{\sigma, y_j} u | + 2 |\p_t \psi| |\p_t u| + |\p_t^2 \psi| |u| \\
& \leqslant | X \mdot \namla u + V_y \, u | + 2 |\p_t \psi| |\p_t u| + |\p_t^2 \psi| |u| \text{.}
\end{align*}
Recalling the bounds for $X$ and $V_y$ (see Assumption \ref{ass.setting}), we then have
\begin{align*}
\sum_{j=1}^2 \int_{\mc{C}} e^{-2\lambda f_j} |P_{\sigma,y_j} (\psi u)|^2 & \leqslant C_1 \sum_{j=1}^2 \int_{I_\delta \times \Omega} e^{-2\lambda f_j} [ | \nabla u |^2 + (\p_t u)^2 + d_\Gamma^{-2} u^2 ] \\
&\qquad + C_1 \delta^{-2} \sum_{j=1}^2 \int_{J_\delta \times \Omega} e^{-2\lambda f_j} [ | \nabla u |^2 + (\p_t u)^2 + d_\Gamma^{-2} u^2 ] \\
& \leqslant C_1 \sum_{j=1}^2 \int_{I_\delta \times \Omega} e^{-2\lambda f_j} [ | \nabla u |^2 + (\p_t u)^2 + u^2 ] \\
&\qquad + C_1 \delta^{-2} \sum_{j=1}^2 e^{ 2 \lambda e^{\alpha_j} } \int_{J_\delta \times \Omega} [ | \nabla u |^2 + (\p_t u)^2 + u^2 ] \text{,}
\end{align*}
where in the last step, we also applied Proposition \ref{lemma_hardy_para} and \eqref{eq_obs_pf_1}.
(Note that the constants $C_1$ can depend on $n$, $\Omega$, $\sigma$, $X$, $V$.)
Recalling also \eqref{eq_energy_defn} yields
\begin{align*}
\sum_{j=1}^2 \int_{\mc{C}} e^{-2\lambda f_j} |P_{\sigma,y_j} (\psi u)|^2 & \leqslant C_1 \sum_{j=1}^2 \int_{I_\delta \times \Omega} e^{-2\lambda f_j} [ |\nabla u |^2 + (\p_t u)^2 + u^2 ] \\
& \qquad + C_1 \delta^{-2} \sum_{j=1}^2 e^{ 2 \lambda e^{\alpha_j} } \int_{J_\delta} E[u] (s) \, \rd s \text{.} \nt \label{eq_obs_pf_4} 
\end{align*}

We now combine the estimates \eqref{eq_obs_pf_3} and \eqref{eq_obs_pf_4}.
Enlarging $\lambda$ so that $C'' \lambda e^{ h_j } \gg C_1$ (by taking $C'' \lambda \gg C_1 e^T$), we can absorb the integral over $I_\delta \times \Omega$ in \eqref{eq_obs_pf_4} into the right-hand side of \eqref{eq_obs_pf_3}.
From all this, we obtain the following estimate:
\begin{align*}
& C_b \mu \lambda \sum_{j=1}^2 \int_{\p\mc{C}} e^{h_j} e^{-2 \lambda f_j} (\mc{N}_q u)^2 + C_1 \delta^{-2} \left( \sum_{j=1}^2 e^{ 2 \lambda e^{\alpha_j} } \right) \int_{J_\delta} E [u] (s) \, \rd s \\
& \quad \geqslant C'' \sum_{j=1}^2 \int_{I_\delta \times \Omega} \lambda e^{h_j} e^{-2\lambda f_j} \big[| \nb_x u|^2 + (\p_t u)^2 + u^2 \big] \text{.}
\end{align*}
Since $(-\delta,\delta) \subseteq I_\delta$, we can reduce the region of integration on the right-hand side,
\begin{align*} 
& C_b \mu \lambda \sum_{j=1}^2 \int_{\p\mc{C}} e^{h_j} e^{-2 \lambda f_j} (\mc{N}_q u)^2 + C_1 \delta^{-2} \left( \sum_{j=1}^2 e^{ 2 \lambda e^{\alpha_j} } \right) \int_{J_\delta} E[u] (s) \, \rd s \\
& \quad \geqslant C'' \sum_{j=1}^2 \int_{(-\delta,\delta) \times \Omega} \lambda e^{h_j} e^{-2\lambda f_j} \big[ | \nabla u|^2 + (\p_t u)^2 + u^2 \big] \\
& \quad \geqslant C'' \lambda \left( \sum_{j=1}^2 e^{ \alpha_j } e^{-2 \lambda (a- e^{\alpha_j})} \right) \int_{- \delta}^{ \delta} E [u] (s) \rd s, \nt \label{eq_obs_pf_z}
\end{align*}
where we used \eqref{eq_obs_pf_1} in the last step.

Also, note that \eqref{eq_energy_est} implies that
\[
e^{-MT} E [u](0) \leqslant E [u](t) \leqslant e^{MT} E [u](0) \text{,} \qquad t \in [-T,T] \text{.}
\]
Using this in \eqref{eq_obs_pf_z}, shows that
\begin{align*}
&C_b \mu \lambda \sum_{j=1}^2 \int_{\p\mc{C}} e^{h_j} e^{-2 \lambda f_j} (\mc{N}_q u)^2 + C_1 \delta^{-1} \left( \sum_{j=1}^2 e^{ 2 \lambda e^{\alpha_j} } \right) e^{MT} E [u](0) \\
& \quad \geqslant C'' \delta \left( \sum_{j=1}^2 \lambda e^{\alpha_j} e^{-2 \lambda (a- e^{\alpha_j})} \right) e^{-MT} E [u](0)\text{.}
\end{align*} 
Finally, by taking $\lambda$ large enough so that $\lambda e^{ \alpha_j } \delta^2 \gg e^{2MT}$ for both $j = 1, 2$, we can absorb the second term in the left-hand side of the above estimate into the right-hand side.
This yields the desired observability estimate and completes the proof of the theorem.

\begin{remark}
One subtle point here is that the constants $C$, $C'$, $C_b$, $\mu_0$ from the Carleman estimate of Theorem \ref{thm_carl_est_main} depend only on $n$, $\Omega$, $\sigma$---but not on $T$.
This allows us to choose $T$ as in \eqref{T_choice}, so that this $T$ always exists and depends on $n$ $\Omega$, $\sigma$.
Once $T$ is chosen, then $\lambda$ can be made even larger in the remaining steps in order to complete the proof.
\end{remark}

\bibliographystyle{amsplain}   % or alpha, plain, amsalpha, etc.
\bibliography{paper_bib}

\end{document}